\documentclass[10pt,twoside]{amsart}
 \usepackage{amsmath,amssymb,amsthm,amscd,latexsym,graphicx}
 \usepackage{epigraph}

 \usepackage[OT2,OT1]{fontenc}

\usepackage{tikz}
\usepackage{graphicx}
\usepackage[all]{xy}
\usepackage{blindtext}
\usepackage{titletoc}

\usepackage{enumitem}

 \font\caps=cmcsc10     
 \font\Caps=cmcsc10 scaled \magstep1 

  \entrymodifiers={!!<0pt,0.7ex>+}

 \makeatletter
 \@specialpagefalse\headheight=8.5pt
 \makeatother

 \def\TSkip{\medskip}
 \newbox\TheTitle{\obeylines\gdef\GetTitle #1
 \ShortTitle #2
 \SubTitle #3
 \Author  #4
 \ShortAuthor #5
 \EndTitle
 {\setbox\TheTitle=\vbox{\baselineskip=20pt\let\par=\cr\obeylines%
 \halign{\centerline{\Caps##}\cr\noalign{\medskip}\cr#1\cr}}%
   \copy\TheTitle\TSkip\TSkip%
 \def\next{#2}\ifx\next\empty\gdef\STitle{#1}\else\gdef\STitle{#2}\fi%
 \def\next{#3}\ifx\next\empty%

 \else\setbox\TheTitle=\vbox{\baselineskip=20pt\let\par=\cr\obeylines%
  \halign{\centerline{\caps##} #3\cr}}\copy\TheTitle\TSkip\TSkip\fi%
  
 \centerline{\caps #4}\TSkip\TSkip%
 \def\next{#5}\ifx\next\empty\gdef\SAuthor{#4}\else\gdef\SAuthor{#5}\fi%
 \catcode'015=5}}

 \long\def\MSC#1\EndMSC{\def\arg{#1}\ifx\arg\empty\relax\else
  {\par\narrower\noindent%
  2000 Mathematics Subject Classification: #1\par}\fi}

 \long\def\KEY#1\EndKEY{\def\arg{#1}\ifx\arg\empty\relax\else
   {\par\narrower\noindent Keywords and Phrases: #1\par}\fi\TSkip}

 \long\def\DATE#1\EndDATE{\def\arg{#1}\ifx\arg\empty\relax\else
   {\par\narrower\noindent \center{\textit{#1}}\par}\fi\TSkip\TSkip\TSkip}

 \font\bf= cmbx10 at 10pt

 \newcommand{\Sym}{\hbox{\ptitgot S}}

 \newcommand{\Ind}{\mathrm{Ind}}

 \newcommand{\Cor}{\mathrm{Cor}}

 \renewcommand{\to}{\longrightarrow}

 \newcommand{\A}{\mathbb A}

 \newcommand{\F}{\mathbb{F}}
 \newcommand{\G}{\mathbb{G}}

 \newcommand{\N}{\mathbb{N}}
 \renewcommand{\P}{\mathbb P}

 \newcommand{\Z}{\mathbb{Z}}

 \newcommand{\Res}{\mathrm{Res}}
 
 \newcommand{\myvdots}{\raisebox{.006\baselineskip}{\ensuremath{\vdots}}}

 \newcommand{\EXT}{\mathrm{Ext}}
 \newcommand{\EEXT}{\mathbf{Ext}}
 \newcommand{\GL}{\mathbf{GL}}

 \newcommand{\Aut}{\mathrm{Aut}}

 \newcommand{\Ext}{\mathrm{Ext}}
\newcommand{\Gal}{\mathrm{Gal}}
 
 \newcommand{\Hom}{\mathrm{Hom}}
 
 \newcommand{\Ver}{\mathrm{Ver}}
 \newcommand{\ver} {\mathrm{ver}}

\newcommand{\Id}{\mathrm{Id}}

\renewcommand{\Im}{\mathrm{Im}}

\newcommand{\Ker}{\mathrm{Ker}}
 
 \newcommand{\Spec}{\mathrm{Spec}}

\newcommand{\nat}{\mathrm{nat}}

 \newcommand{\vers}{\mathrm{vers}}

\newcommand{\W}{\mathbf{W}}
\newcommand{\E}{\mathcal{E}}
\newcommand{\frob}{\mathrm{frob}}

\newcommand{\EExt}{\mathbf{Ext}}

\renewcommand{\Sym}{\mathrm{Sym}}

 \theoremstyle{plain}
 \newtheorem{thm}{Theorem}[section]
   \newtheorem*{thm*}{Theorem}
 \newtheorem{defi}[thm]{Definition}
  \newtheorem{nota}[thm]{Notation}

 \newtheorem{prop}[thm]{Proposition}

 \newtheorem{lem}[thm]{Lemma}

 \newtheorem*{theorem-non}{Bloch-Kato-Milnor conjecture, an equivalent formulation}
\newtheorem*{thmA*}{Theorem A}
\newtheorem*{defi*}{Definition}
\newtheorem*{thmB*}{Theorem B}
\newtheorem*{thmC*}{Theorem C}
\newtheorem*{thmD*}{Theorem D}
\newtheorem*{thmL*}{Theorem L}

 \theoremstyle{remark}
 \newtheorem{rem}[thm]{Remark}
 \newtheorem{qu}[thm]{Question}

 \newtheorem{ex}[thm]{Example}
 
 \newtheorem{exo}[thm]{Exercise}

 \newenvironment{dem}{{\bf Proof.}}{\hfill$\square$}

\newcommand{\fonction}[5]{\begin{array}{cccc}
#1: & #2 & \longrightarrow & #3 \\
 & #4 & \longmapsto & #5 \end{array}}
 
\newcommand{\fonctionnoname}[4]{\begin{array}{ccc}
 #1 & \longrightarrow & #2 \\
 #3 & \longmapsto & #4 \end{array}}

\date{March 2025}

\usepackage{titletoc}

\usepackage{tikz}
\usetikzlibrary{arrows, decorations.markings}
\tikzstyle{vecArrow} = [thick, decoration={markings,mark=at position
   1 with {\arrow[semithick]{open triangle 60}}},
   double distance=1.4pt, shorten >= 5.5pt,
   preaction = {decorate},
   postaction = {draw,line width=1.4pt, white,shorten >= 4.5pt}]
   
   \tikzstyle{vec} = [thick, decoration={markings,mark=at position
   1 with {\arrow[semithick]{}}},
   double distance=1.4pt, shorten >= 7.1pt,
   preaction = {decorate},
   postaction = {draw,line width=1.4pt,white, shorten >= 3.5pt}]
\tikzstyle{innerWhite} = [semithick, white,line width=1.7pt, shorten >= 4.5pt]

\begin{document}
\title{Smooth profinite groups, I: geometrizing Kummer theory}

\author{Charles De Clercq and Mathieu Florence}
\address{Charles De Clercq, Équipe Topologie Alg\'ebrique, Laboratoire Analyse, G\'eom\'etrie et Applications, Sorbonne Paris Nord, 93430 Villetaneuse, France.}
\address{Mathieu Florence, Équipe topologie et géométrie algébriques, Institut de Mathématiques de Jussieu, Sorbonne Université, F-75005 Paris, France.}


\begin{abstract}
In this series of three papers, we introduce and study cyclotomic pairs and smooth profinite groups. They are a  geometric axiomatisation of Kummer theory for fields, with coefficients  $p$-primary roots of unity, for a prime $p$. These coefficients are enhanced, to  $G$-linearized line bundles in Witt vectors, over $G$-schemes of characteristic $p$. In the second paper \cite{F2}, this upgrade is pushed even further, to the scheme-theoretic setting.\\
In this first article, we introduce cyclotomic pairs, smooth profinite groups and $(G,S)$-cohomology. We prove a first lifting theorem for $G$-linearized torsors under line bundles (Theorem A).

With the help of the algebro-geometric tools developed in the second article,  this formalism is applied in the third one \cite{DCF3}, to prove the Smoothness Theorem, whose essence reads as follows. Let $G$ be  profinite group. Assume that, for every open subgroup $H \subset G$, and for $n=1$, the natural arrow $H^n(H,\Z/p^2) \to H^n(H,\Z/p)$ is surjective. Then, it is also surjective for every such $H$, and every $n \geq 2$.  Applied to absolute Galois groups, the Smoothness Theorem provides a new proof of the Norm Residue Isomorphism Theorem, entirely disjoint from motivic cohomology.\\
\end{abstract}

\maketitle

\tableofcontents
\newpage
\section{Introduction.}

Let $m$ be a positive integer and let $F$ be a field where $m \neq 0$. Fix a separable closure $F_s/F$ and denote by $\mu_m$ the Galois module of $m$-th roots of unity in $F_s$. Kummer theory, in its most elementary form, states the following. Consider the  Kummer exact sequence
$$1\to \mu_m\to F_s^{\times}\overset{(\cdot)^{m}}{\to} F_s^{\times}\to 1.$$
Then, the Bockstein homomorphism, i.e. the associated boundary map $$\delta^1_{F,m}:F^{\times}\to H ^1(F,\mu_m),$$ 
is surjective by  Theorem 90, with kernel $F^{\times m}$.

For cohomology groups $H^n(F,\mu_m^{\otimes n})$ of degree $n>1$, producing a description of this kind, through tensor products of copies of $F^{\times}$ with appropriate relations, is a much more difficult problem. For this purpose, inspired by the Steinberg relations in Matsumoto's description of  $K_2(F)$, Milnor introduced in the sixties the $K$-theory groups $K^M_n(F)$. Milnor, Bass and Tate then extended the Bockstein $\delta^1_{F,m}$ above to morphisms $$h^n_{F,m}:K^M_n(F)\to H^n(F,\mu_m^{\otimes n})$$ called Galois symbol (or  norm residue map). They question whether these yield isomorphisms $$K^M_n(F)/m\overset{\sim}{\longrightarrow} H^n(F,\mu_m^{\otimes n}),$$for any field $F$ in which $m$ is invertible. This is the statement later known as the Bloch-Kato-Milnor conjecture.

A major breakthrough towards this conjecture was achieved in 1982 by Merkurjev and Suslin, who solved it for $n=2$ \cite{MS}. In 1996, Voevodsky proved the case where $m$ is a power of $2$. After tremendous efforts, a full proof of the conjecture was completed in 2008 by Rost, Suslin, Voevodsky, and Weibel \cite{HW}. The statement is now known as the Norm Residue Isomorphism Theorem. 

In this series of three articles, we are interested in studying structural properties of \emph{smooth profinite groups}, leading notably to a generalization of the following statement, equivalent to the Norm Residue Isomorphism Theorem (see \cite{Gi}, \cite{Me}).

\begin{theorem-non} \quad\\
Let $F$ be a field and let $p$ be a prime, invertible in $F$. Then, the Bockstein $$H^n(F,\mu_p^{\otimes n})\longrightarrow H^{n+1}(F,\mu_p^{\otimes n})$$
associated with the exact sequence $$1\to \mu_p^{\otimes n} \to \mu_{p^2}^{\otimes n}\to \mu_p^{\otimes n}\to 1$$
is trivial, for any positive integer $n$.\\
Equivalently, the induced map $$H^n(\Gal(F_s/F), \mu_{p^2}^{\otimes n})\to H^n(\Gal(F_s/F), \mu_p^{\otimes n})$$ is surjective.
\end{theorem-non}

A sleek aspect of this statement is that it only involves Galois cohomology: forgetting $K$-theoretic considerations, the only characters on stage are the profinite group  $\Gal(F_s/F)$ and  the Galois module $\mu_{p^2}$. Taking account of finite separable extensions $E/F$, the property readily follows replacing $\Gal(F_s/F)$ by the open subgroups $\Gal(F_s/E) \subset \Gal(F_s/F)$. This fact is the initial motivation for introducing the following notion.

\begin{defi*}(Cyclotomic pair,  Definition \ref{cyclotomicpair})\\
    Let $n,e$ be positive integers and let $G$ be a profinite group. Let $\mathcal{T}$ be  a free $\mathbb{Z}/p^{e+1}\mathbb{Z}$-module  of rank $1$, endowed with a continuous action of $G$. We say that the pair $(G,\mathcal{T})$ is $(n,e)$-cyclotomic if for any open subgroup $H\subset G$, the natural morphism $$H^n(H,\mathcal{T}^{\otimes n})\to H^n(H,(\mathcal{T}/p)^{\otimes n})$$ is surjective.
\end{defi*}

The above thus gives a direct analogy: Kummer theory implies that, setting $$\mathcal{T}:=\varprojlim_r \mu_{p^r}$$ to be the usual Tate module, the pair $(\Gal(F_s/F),\mathcal{T})$ is $(1,\infty)$-cyclotomic.

Henceforward, the above version of the Bloch-Kato-Milnor conjecture can be translated, in the broader context of cyclotomic pairs, as

\begin{center}
 \hspace{-0.5cm} ($\ast$) \hspace{0.3cm}If $(G,\mathcal{T})$ is a $(1,1)$-cyclotomic pair, then it is $(n,1)$-cyclotomic for every $n \geq 2$.
\end{center}
This statement is  the Smoothness Conjecture \cite[Conjecture 14.25]{DCFOr}, in depth $e=1$. It fleshes out the common belief that the keystone of the Norm Residue Isomorphism Theorem, is  Theorem 90 for fields.\\ 
In this work, we prove $(\ast)$, yielding   the Smoothness Theorem (Theorem B). Thus,   the above version of the Bloch-Kato-Milnor Conjecture (vanishing of higher Bockstein), is true in the context of $(1,1)$-cyclotomic pairs. This  applies to a much broader context, than Galois cohomology:  among other examples,   \'etale fundamental groups of connected curves over algebraically closed fields and of semilocal rings, fit in $(1,1)$-cyclotomic pairs \cite[§4]{DCF0} (see also \cite[Appendix]{DCF3}).

\subsection{A failed  attempt.} In a previous version of this trilogy, we  raised a bridge between   $(\ast)$, and the liftability mod $p^2$, of mod $p$ representations of $G$. For absolute Galois groups, this 
reads as follows.
\begin{qu}\label{QuGalLift} \hfill \\
   Let $F$ be a field, with separable closure $F_s$. For $d \geq 1$, consider a representation $$\rho_1:\Gal(F_s/F)\to \GL_d(\mathbb{F}_p).$$ Is there a  representation $\rho_2$, such that the diagram $$\xymatrix@R=2em {\Gal(F_s/F)\ar[r]^-{\rho_2}\ar[rd]_-{\rho_1}&\GL_d(\mathbb{Z}/p^2\mathbb{Z})\ar[d]\\&\GL_d(\mathbb{F}_p)}$$
commutes? [The vertical arrow is given by  reduction.]
\end{qu}

In  a previous \emph{incorrect} version of the second article of this trilogy, Mathieu Florence thought he  had proved the so-called \emph{Uplifting Theorem}, stating that the answer to (a much more evolved triangular form of) Question \ref{QuGalLift} be positive. In a  previous correct, but now useless version of the third paper, it is shown that this  `Uplifting Theorem' implies the Smoothness Theorem. As the upshot of efforts to disprove this approach,  Merkurjev and Scavia  \cite{MeSc} settled that the answer  to Question \ref{QuGalLift} is positive \textit{for every} $F$, if and only if $p=2$ and $d \leq 4$, or $p\geq 3$ and $d \leq 2$.\\
To prove   Theorem B, a  down-to-earth lifting process was then designed to replace the erroneous `Uplifting Theorem' : the Uplifting Pattern of \cite{F2}.

\subsection{Our pathway.}

A crucial point is to design more flexible objects. Indeed, as stated above, Kummer theory (say, for $m=p^r$) has an obvious weakness: whereas it holds for any field $F$, its coefficients are  invariably $\mu_{p^r}$. To fix this, we  `geometrize' these coefficients- upgrading them to $G$-linearized line bundles in $p$-typical Witt vectors of length $r$, on $G$-schemes characteristic $p$. Upon replacing $\mu_p$ by a $G$-linearized line bundle $L$, the analogue of  $\mu_{p^r}$ is the Teichmüller lift $\W_r(L)$ (up to a twist). Equivariant Witt vector bundles and their extensions, whose systematic study was initiated in \cite{DCFL}, thus plays a key role in our approach.

In this first article, we  prove a general lifting theorem, to be thought of as a  geometrization of classical Kummer theory- replacing $\mu_ {p}$ by an arbitrary $G$-line bundle, over a $G$-scheme $S$ of characteristic $p$. We give here its flavor.

\begin{thmA*}(Particular case; see §8 for the full statement.)\\ Let $G$ be a profinite group. Let $\Z/p^2(1)$ be a discrete $G$-module, that is a free $\Z/p^2$-module of rank one.\\
Let $S$ be a perfect affine $(G,\F_p)$-scheme. Let $L$ be a $G$-line bundle over $S$.\\ Assume that, for every open subgroup $H \subset G$, the natural arrow \[ H^1(H,\W_2(L)(1)) \to H^1(H,L(1)) \] is surjective, in the particular case $S=\Spec(\F_p)$ and $L=\F_p$. \\
Then, it is surjective for every $S,L$ and $H$ as above.
\end{thmA*}

[A pair $(G,\Z/p^2(1))$ as in the premises of Theorem A, is called $(1,1)$-cyclotomic.]\\

A word of explanation is needed, concerning the new notion of a $(n,e)$-smooth profinite group $G$, given in Definition \ref{DefiSmooth}. At its core lies  even more flexibility. We say that $G$ is  $(n,e)$-smooth if  the following holds.\\
Let $L_1$ be a $G$-linearized line bundle over a perfect affine $(G,\F_p)$-scheme $S$ and let $$c_1 \in H^n((G,S),L_1)$$  be a cohomology class. Then $c_1$ lifts to a class $$c_{e+1} \in H^n((G,S),L_{e+1}[c_1]),$$ for \textit{some} $G$-linearized invertible  $\W_{e+1}(A)$-module $L_{e+1}[c_1]$,  \emph{depending} on  $c_1$.

Smoothness is thus intrinsic to the profinite group $G$- whereas a  cyclotomic pair obviously depends on the given cyclotomic module $\Z/p^{1+e}(1)$. If $(G,\Z/p^{1+e}(1))$ is a $(n,e)$-cyclotomic pair, Theorem A implies that $G$ is $(n,e)$-smooth.\\

The goal of the second article is to  set up the Uplifting Pattern. Roughly speaking, it is  a process to lift equivariant extensions of vector bundles, to their $\W_2$-counterparts, upon  combined \emph{group-change} and \emph{base-change}.  For this to make sense, one first needs to define a suitable framework. This leads to an \textit{`integral scheme-theoretic enhancement'} of the key ideas developed in the first article. In short: in \cite{F2}, one typically works over a $p$-torsion-free base (e.g. $\Spec(\Z)$, or smooth $\Z$-schemes) rather than over perfect $\F_p$-schemes. Also, actions of finite groups, are upgraded to actions of smooth affine $\Z$-group schemes (these involve  transfer w.r.t.   truncated Witt vector ring schemes). \\The \emph{group-change} is the so-called \emph{suspension}, designed to lift Hochschild $1$-cocycles \cite[§13]{F2}. In a suitable sense, suspension provides the scheme-theoretic  analogue of the cyclotomic pairs considered in this paper. In this integral framework, \textit{liftability} w.r.t. the ring scheme $\W_2$ would be insufficient. This reflects the fact that, in Theorem A (full form), it is necessary to apply a frobenius twist prior to lifting (see Section \ref{secperfnec}). Integrally, such a twist does not exist! One way to solve this difficulty,  is to replace    $\W_2$ by  a more flexible ring scheme $\W_2^{[r]}$, depending on a given integer $r \geq 0$. It is defined over $\Z$, and built out of $\W_{r+2}$. It is specifically taylored to act as `the $r$-th Frobenius twist' of $\W_2$. \\ The \emph{base-change} is the so-called Uplifting scheme (see \cite[§12]{F2}). It is the torsor of lifts of the frobenius of (mod $p$) vector bundles. \\ The Uplifting Pattern is the key ingredient to prove the  Smoothness Theorem, in the third article (see \cite[§4]{DCF3} for the full statement).

\begin{thmB*}[Smoothness theorem] \hfill \\
Keep notation and assumption of Theorem A. The arrow \[ H^n(H,\W_2(L)(n)) \to H^n(H,L(n)) \] is then surjective, for every $S,L$,$H$, and $n \geq 1$.
\end{thmB*}

Here  is a sketch of the proof. By (the full form of) Theorem A, one may assume $S=\Spec(\F_p)$ and $L=\F_p$. Let $c \in H^n(G,\F_p(n))$ be given. The goal is to lift $c$, to some  $c_2 \in H^n(G,\Z/p^2(n))$. Via suspension, we construct a \emph{versal}  class $C_{\vers}$ (in  Hochschild cohomology of  a suitable affine smooth $\Z$-group scheme $\Gamma_{\vers}$), defined over a projective space over $\Z$. The  class  $C_{\vers}$ specialises  to $c$, upon suitable base/group-change (this  uses that $(G,\Z/p^2(1))$ is a $(1,1)$-cyclotomic pair). The Smoothness theorem  then boils down to lifting  $C_{\vers}$: specialising a  versal lift, yields the sought-for $c_2$.   For purely geometric reasons (related to lifting the mod $p$ frobenius of vector bundles), $C_{\vers}$ lifts after base-change, via the  so-called `Uplifting Scheme'.  For purely group-theoretic reasons (computations in Hochschild cohomology of linear algebraic groups), this base-change can then be undone: $C_{\vers}$ indeed lifts, prior to base-change.

\vspace{0.2cm}
We give a Leitfaden, connecting significant results of our three papers. For simplicity, we stick to Galois cohomology. Recall that the same diagram holds, replacing absolute Galois groups by $\pi_1(X)$, where $X$ is a curve over an algebraically closed field, or a semilocal $\Z[ \frac 1 p]$-scheme. \vspace{0.5cm}

\begin{center}
\begin{tikzpicture}[thick]

\node[draw,rectangle,align=center,below of=k] (a) at (2.25,-0.5) {Cyclotomic pairs and\\Smooth profinite groups};
\node[draw,rectangle,align=center] (b) at (2.25,-3.9) {Absolute Galois groups\\are $(1,\infty)$-smooth};
\node[draw,rectangle,align=center] (c) at (2.25,-6.5) {Norm Residue\\Isomorphism Theorem};
\node[draw,rectangle,align=center] (d) at (9,-1.5) {Uplifting pattern\\
(Smooth Profinite groups II)};
\node[draw,rectangle,align=center] (e) at (9,-3.9) {Smoothness Theorem\\
(Smooth Profinite groups III)};
\node[draw,rectangle,align=center] (f) at (9,-6.5) {Absolute Galois groups\\are $(n,1)$-smooth};

\node[align=center] (g) at (1.2,-2.65) {\footnotesize Theorem A};
\node[align=center] (x) at (5.42,-0.8) {\footnotesize scheme-theoretic};
\node[align=center] (y) at (5.42,-1.1) {\footnotesize enhancement};
\node[align=center] (z) at (10.8,-2.65) {\footnotesize Hochschild cohomology};

 \draw[vecArrow] (a) to (d);
 \draw[vec] (4.138,-3.9) to (6.85,-3.9);
 \draw[vecArrow] (a) to (b);
 \draw[vecArrow] (9,-2.008) to (e);
 \draw[vecArrow] (e) to (f);
 \draw[vecArrow] (f) to (c);
 \end{tikzpicture}
\end{center}

\vspace{.5cm}


For the reader's convenience, an  index is available, which tables  various notions introduced in this work.



   

\section{$G$-equivariant constructions.}
\subsection{General setting.}
Let $X$ be an object of a category $\mathcal{C}$, and $G$ be a profinite group. In this text, a \emph{naive action} of $G$ on $X$ is an action of the abstract group $G$ on $X$, whose kernel $G_0$ is an open subgroup of $G$. We denote by $G- \mathcal C$ the category whose objects are objects of $\mathcal C$, equipped with a naive action of $G$, and whose morphisms are the same as in $\mathcal C$. In $G- \mathcal C$, $\Hom$-sets are actually enriched with the structure of $G$-sets. Thus, $G$-equivariant morphisms $X\to Y$, between objects of $G- \mathcal C$, are fixed elements of the $G$-set $\Hom(X,Y)$. In short: \[ \Hom_{G-equ}(X,Y) =H^0(G,\Hom(X,Y)).\]
An object of $G- \mathcal C$ will be called a $G$-object of $\mathcal C$.

\begin{rem}
In the sequel unless specified otherwise, all actions are assumed to be naive.
\end{rem}
\subsection{$G$-linearized modules over $G$-schemes.}
In this work, all schemes are assumed to be quasi-compact and quasi-separated.   Sheaves are with respect to Zariski topology. We  restrict to `topologically well-behaved' $G$-actions, in the sense of the following Definition. 

\begin{defi}\index{$G$-scheme}
A $G$-scheme (or scheme with a $G$-action) is the data of a scheme $S$, equipped with a naive action of $G$, satisfying the property:

\hspace{1cm}\emph{(}$\ast$\emph{)} $S$ is covered by affine $G$-invariant open subschemes.

The collection of all $G$-schemes form a category $G-Sch$, with morphisms being usual morphisms of schemes.\\
A $(G,\F_p)$-scheme is a $G$-scheme of characteristic $p$.\\
If $S$ is a given $G$-scheme, a $(G,S)$-scheme is a $G$-equivariant morphism $$T \to S,$$ in $G-Sch$. It will most of the time be simply denoted by $T$.
\end{defi}

\begin{rem}
In general, $G$ may act on a scheme $S$, in such a way that $S$ is \textit{not} covered by affine $G$-invariant open subschemes. See, however, the next Exercise.
\end{rem}
\begin{exo}
Let $S$ be a scheme, separated over $\Z$, such that every finite set of points of $S$ is contained in an open affine subscheme of $S$. Show that $S$ has property $(*)$, for any naive action of $G$ on $S$.
\end{exo}

It is clear that a closed subscheme of a $G$-scheme, given by a $G$-invariant ideal, is a $G$-scheme as well. This also holds for open subschemes.
\begin{lem}
Let $S$ be a $G$-scheme. Let $U \subset G$ be a $G$-invariant open subscheme. Then, $U$ is a $G$-scheme as well.
\end{lem}

\begin{dem}
We can assume that $S=\Spec(A)$ is affine, and $G$ finite. The complement of $U$ in $S$ is given by a $G$-invariant ideal $I \subset A$. Pick a point $u \in U$. Denote by $P_1, \ldots, P_n$ the distinct prime ideals of $A$ corresponding to the $G$-orbit of $u$. For each $i=1 \ldots n$, there exists an element $a_i \in I$, not belonging to $P_i$ but belonging to all other $P_j$'s. Put $a:= \sum_1^n a_i$. Then, the principal open set $D(a)$ is contained in $U$, and contains the $G$-orbit of $u$. Denoting by $$f:= \prod_{g \in G} g\cdot a$$ the norm of $a$, we see that $D(f) \subset U$ is an affine $G$-invariant open, containing $u$. Thus, $U$ can be covered by affine $G$-invariant open subschemes.
\end{dem}

The next simple definition suffices for our purposes.
\begin{defi}
 Let $S$ be a $G$-scheme. A $G$-presheaf on $S$, with values in a category $\mathcal D$, is a contravariant functor, from the category of $G$-invariant open subsets of $S$ (where morphisms are inclusions), to $\mathcal D$. A $G$-sheaf is a $G$-presheaf, satisfying the usal sheaf axiom.\\
 In most applications,  $\mathcal D$ will actually be  $G-\mathcal C$,  where  $\mathcal C$ is a category.
\end{defi}
\begin{defi}\index{$G$-linearized $\mathcal O_S$-module}
Let $S$ be a $G$-scheme. A $G$-linearized $\mathcal O_S$-module is the data of a quasi-coherent $\mathcal O_S$-module $M$, equipped with a continuous semilinear action of $G$. In concrete terms, such an action is given by isomorphisms of $\mathcal O_S$-modules \[\phi_g: M \to (g.)^*(M), \] one for each $g \in G,$ such that the following conditions hold :

\begin{enumerate}[label=\roman*)]
\item The mapping $g \mapsto \phi_g$ is locally constant on $G$, i.e. factors through a quotient $G \to G/G_0$, by a normal open subgroup.

\item  We have \[ \phi_{gh}= (h.)^*(\phi_g) \circ \phi_h, \] for each $g,h \in G.$

\end{enumerate}

We  say $(G,\mathcal O_S)$-module, or $(G,S)$-module, instead of $G$-linearized $\mathcal O_S$-module.\\
The collection of all $(G,\mathcal O_S)$-modules form an Abelian category, monoidal through the tensor product $\otimes=\otimes_{\mathcal O_S}$. We denote it by $(G,\mathcal O_S)-Mod$.\\
If $M$ and $N$ are two $(G,\mathcal O_S)$-modules, the internal Hom of $\mathcal O_S$-modules $\underline \Hom_{\mathcal O_S}(M,N)$ is naturally a $(G,\mathcal O_S)$-module, which we denote simply by $\underline \Hom(M,N)$. We put $$M^\vee:= \underline \Hom(M,\mathcal O_S).$$
A locally free $(G,\mathcal O_S)$-module of finite constant rank as an $\mathcal O_S$-module, will be called a $G$-vector bundle on $S$.

\end{defi}

\begin{rem}
In the previous Definition, the largest open subgroup through which $g \mapsto \phi_g$ factors may be much smaller than the kernel of the action of $G$ on $S$.
\end{rem}

\begin{rem}
In short, a $(G,\mathcal O_S)$-module is the data of a quasi-coherent $\mathcal O_S$-module, equipped with a semilinear (naive) action of $G$.\\
For $G$ finite, a $G$-line bundle is a $G$-linearized line bundle over $S$, in the sense of Geometric Invariant Theory \cite[Chapter 1 §3]{Mum}.
\end{rem}

\begin{rem}
Assume that $X=\Spec(A)$ is an affine $G$-scheme. In other words, $A$ is a commutative ring, endowed with a naive action of $G$. We then  use the denomination $(G,A)$-module (resp.  $(G,A)$-bundle) for a $(G,\mathcal O_S)$-module (resp. a $(G,\mathcal O_S)$-bundle). A  $(G,A)$-module is the data of an $A$-module $M$, equipped with a semilinear action of $G$. The `semi' part of linearity means that $$g.(am)=g(a).g(m),$$ for all $g \in G,$ $a \in A$ and $m \in M$.\\  In particular, if $G$ is the absolute Galois group of a field $F$ and if $A=\F_p$, a $(G,A)$-module  is a Galois representation of the field $F$, with $\mathbb{F}_p$ coefficients.
\end{rem}
\begin{rem}
Let $G$ be a finite group, acting on a  ring $A$. Consider  two  cases.
\begin{enumerate}
    \item The $G$-action on  $A$ is trivial. Then, a $(G,A)$-module $M$ such that $M =A^d$ as an $A$-module, is the same thing, as a continuous representation $$ \rho: G \to \GL_d(A).$$ 
    
    \item The $G$-action on $\Spec(A)$ is free.  Set $B:=H^0(G,A)$. Then, $B/A$ is a $G$-Galois algebra. By  Speiser's Lemma (\cite{GS}, Lemma 2.3.8), the category of $(G,A)$-modules is equivalent to that of $B$-modules, via the assignment  
    $$\fonctionnoname{\{B-Mod\}}{\{(G,A)-Mod\}}{N}{A \otimes_B N}$$    
 with quasi-inverse 
   $$\fonctionnoname{\{(G,A)-Mod\}}{\{B-Mod\}}{M}{H^0(G,M)}.$$    
\end{enumerate}
Attempting to prove a natural property satisfied by $(G,A)$-modules for arbitrary $G$-actions on $A$, it is advisable to check it first in these two extreme cases. If it does, the property is likely to hold true in general.
\end{rem}

\begin{rem}\label{Ginvmodules}
Let $S$ be a $G$-scheme, and let $M$ be a quasi-coherent $\mathcal O_S$-module. If $M$ can be $G$-linearized (that is, $M$ can be endowed with a structure of $(G,\mathcal O_S)$-module), then  \textit{the isomorphism class} of $M$  is $G$-invariant. In other words, $M$ is isomorphic to $g^*(M)$, for all $g \in G$. Note that $G$-invariant modules are not always $G$-linearizable-- except, for instance, when $G$ is a free profinite group.
\end{rem}

\section{Recollections on Witt vectors and Witt modules.}
In this first paper, \textit{Witt vectors are considered over $\F_p$-schemes only.}\\
Let $A$ be a ring of characteristic $p$. We denote by $\W(A)$ the ring of $p$-typical Witt vectors built out of $A$. Set-wise, $\W(A)$ is simply $A^{\mathbb{N}}$ with the ring structure derived from universal Witt polynomials (see \cite{SeCL}). For a thorough exposition of  Witt modules and Witt vector bundles, see  \cite{DCFL} and \cite{DCFOr}, where an alternative construction of Witt vectors is provided, using divided powers of abelian groups.

The ring of Witt vectors $\W(A)$ is endowed with a Verschiebung (additive) morphism
$$\fonction{\Ver}{\W(A)}{\W(A)}{(a_0,a_1,a_2,...)}{(0,a_0,a_1,a_2,...)}$$

and the frobenius morphism $\frob:(a_0,a_1,...)\mapsto (a_0^p,a_1^p,...)$.

For any $r\geq 1$, denote by $\W_r(A)$ the ring of truncated Witt vectors of length $r$. We have $\W_1(A)=A$, and the ring $\W(A)$ is the projective limit of the $\W_r(A)$ through the quotient maps
$$\fonction{\pi_{r+1,r}}{\W_{r+1}(A)}{\W_r(A)}{(a_0,...,a_{r+1})}{(a_0,...,a_r)}$$
More generally, for any two integers $r\leq s$, we denote by $\pi_{s,r}$ the quotient map $$\W_s(A)\to \W_r(A).$$ We will often use the following fundamental property: the quotient $$\W(A)\to \W_1(A)=A$$ has a multiplicative section given by the Teichmüller representative $$\tau: a\mapsto (a,0,...),$$ refered to as the multiplicative (or Teichmüller) section.

Consider now a scheme $S$ of characteristic $p$, covered by affine open subschemes $\Spec(A_i)$. We denote by $\W_r(S)$ the scheme of Witt vectors of $S$ of length $n$. It is defined by gluing the affine schemes $\Spec(\W_r(A_i))$ and is a universal thickening of $S$ of order $n$, through the nilpotent closed immersions $\W_r(S) \to \W_{r+1}(S)$. In particular, the underlying topological space of $\W_r(S)$ corresponds to that of $S$.

The following definition is classical.

\begin{defi}[\cite{SeTcarp}]
Let $r\geq 1$ be an integer. The association $$U\mapsto \W_r(\mathcal{O}_S(U))$$ defines a sheaf of (commutative) rings on $S$, denoted by $\W_r(\mathcal{O}_S)$.

By definition, $\W_1(\mathcal{O}_S)$ is simply the structure sheaf $\mathcal{O}_S$ of $S$ and following the previous notations, for $s\geq r$, we denote by $$\pi_{s,r}:\W_s(\mathcal{O}_S)\to \W_r(\mathcal{O}_S)$$ the natural transformation obtained through the morphisms $$\pi_{s,r}(U):\W_s(\mathcal{O}_S(U))\to \W_r(\mathcal{O}_S(U))$$ defined above.
\end{defi}

Witt modules are meant to provide a generalization of quasi-coherent $\mathcal{O}_S$-modules, in higher $p$-primary torsion.

\begin{defi} \label{WtFaffine}
Assume that $S=\Spec(A)$ is affine. Let $r \geq 1$ be a positive integer and let $M$ be a $\W_r(A)$-module. The formula $$ U \mapsto M \otimes_{\W_r(A)} \W_r(\mathcal O_S(U))$$ defines a presheaf on $S$,  for the Zariski topology. We denote by $\tilde M$ the associated sheaf. It is a sheaf of $\W_r(\mathcal O_S)$-modules.
\end{defi}

\begin{defi}[Witt modules]\hspace{0.1cm}\\
A Witt module of height $r\geq 1$ over $S$ is a sheaf of $\W_r(\mathcal O_S)$-modules, which is locally isomorphic to a sheaf of the shape $\tilde M$ (cf. Definition \ref{WtFaffine}).\\When no reference to its height is necessary, a Witt module will simply be referred to as a $\W$-module.\\
A $\W$-module over $S$ locally isomorphic to $\W_r(\mathcal{O}_S)^d$ for some $d \geq 0$ is called a $\W_r$-bundle of rank $d$.
\end{defi}
\subsection{Reduction.}
Let $0\leq r \leq s$ be   integers. Let $\mathcal{F}$ be a sheaf of $\W_s(\mathcal{O}_S)$-modules over $S$. The reduction of $\mathcal{F}$ to  a sheaf of $\W_r(\mathcal{O}_S)$-modules, is the sheaf associated to the presheaf $$U\mapsto \mathcal{F}(U)\otimes_{\W_s(\mathcal{O}_S(U))}\W_r(\mathcal{O}_S(U)),$$ which we denote by $\mathcal{F} \otimes_{\W_s}\! \W_r $.

\subsection{frobenius.}
By functoriality, the absolute frobenius morphism $$\frob:S\to S$$ lifts to an endomorphism of $\W_r(S)$, the frobenius endomorphism of $\W_r(S)$, which we still denote by $\frob$. If $\mathcal{F}$ is a $\W_r$-module over $S$ and $m\geq 0$, we put $$\mathcal{F}^{(m)}:=(\frob^m)^{\ast}(\mathcal{F});$$ it is a $\W_r$-module over $S$. If  $\mathcal{F}$ is a $\W_r$-bundle, then $\mathcal{F}^{(m)}$ is a $\W_r$-bundle as well, of the same rank as $\mathcal{F}$.

In this paper, the frobenius pullback of a $\W_r$-module is always taken w.r.t. the absolute frobenius of the base where it is defined.

\section{$(G,M)$-torsors and Yoneda extensions.}

Let $G$ be a profinite group. Let $A$ be a $G$-group, i.e. a group equipped with a naive action of $G$. There is a well-known bijection between the set $H^1(G,A)$ and isomorphism classes of $G$-equivariant principal homogeneous spaces (torsors) of $A$. In this section, this correspondence is extended to the context of $G$-equivariant torsors under $(G,\mathcal{O}_S)$-modules. A main tool (see Prop. \ref{equaffext}) is  Yoneda's interpretation of torsors, as equivalence classes of extensions \cite{Y}.

\subsection{Extensions and operations on them.}\label{YonedaExt}

Let $S$ be a $G$-scheme, let $n \in \mathbb{N}^{\ast}$ be an integer and let $A,B$ be $(G,\mathcal O_S)$-modules over $S$. As in any additive category, we can consider the notion of  (Yoneda) $n$-extensions of $A$ by $B$, which we now recall (see \cite[§2]{DCF0}). One could work in the famework of derived categories instead.\\

As usual, set $$\Ext_{(G,\mathcal O_S)}^0(A,B):= \Hom_{(G,\mathcal O_S)}(A,B).$$
For $n \geq 1$, an $n$-extension of $A$ by $B$ is an exact sequence of $(G,\mathcal O_S)$-modules $$\E: 0 \to B \to A_1 \to \ldots \to A_{n}\to A \to 0.$$ 
The collection of all these extensions, is denoted by $\EExt_{(G,\mathcal O_S)}^n(A,B)$.

\subsection{Morphisms.} A morphism $$\E_1\to \E_2,$$ between two $n$-extensions of $A$ by $B$, is a morphism of complexes which is the identity on both $A$ and $B$. With these,  $\EExt_{(G,\mathcal O_S)}^n(A,B)$ is a category.

\subsection{Pushforwards and pullbacks.} A morphism $f:B \to B'$ induces a pushforward functor $$f_*: \EExt_{(G, \mathcal O_S)}^n(A,B) \to \EExt_{(G,\mathcal O_S)}^n(A,B').$$ Likewise, a morphism  $g:A' \to A$ induces a pullback functor $$g^*: \EExt_{(G,\mathcal O_S)}^n(A,B) \to \EEXT_{(G, \mathcal O_S)}^n(A',B).$$ The two composite functors $$f_* g^*: \EEXT_{(G,\mathcal O_S)}^n(A,B) \to \EEXT_{(G, \mathcal O_S)}^n(A',B')$$  and  $$g^* f_*: \EEXT_{(G,\mathcal O_S)}^n(A,B) \to \EEXT_{(G, \mathcal O_S)}^n(A',B')$$  are canonically isomorphic.\\
\subsection{Baer sum.}

One can add two $n$-extensions $$\E_1, \E_2 \in \EEXT_{(G,\mathcal O_S)}^n(A,B), $$ using the Baer sum.\\Denoting by $$\fonction{\delta}{A}{A \oplus A}{a}{(a,a)}$$ the diagonal, and by  $$\fonction{\alpha}{B \oplus B}{B}{(b_1,b_2)}{b_1+b_2}$$ the addition, the  formula is $$\E_1+ \E_2 := \alpha_*(\delta^*(\E_1\oplus \E_2)).$$ For this operation, the trivial $1$-extension is the direct sum $$0 \to B \to B \oplus A \to A \to 0.$$ If $n\geq 2$, the trivial $n$-extension is $$ 0 \to B \stackrel {\Id }\to B \to 0 \to \ldots \to 0\to A \stackrel {\Id }\to A \to 0.$$  Baer sum is  $\mathcal O_S$-linear, in the natural fashion.
\subsection{Change of the base.}
Let  $$h: T \to S$$ be a $G$-equivariant morphism of $G$-schemes. Let $$\E: 0 \to B \to A_1 \to \ldots \to A_{n}\to A \to 0$$ be an $n$-extension, in $ \EEXT_{(G, \mathcal O_S)}^n(A,B)$. We would like to define \[ h^*(\mathcal E) \in \EEXT_{(G,\mathcal O_T)}^n(h^*(A),h^*(B)), \]  as the usual base-change  $$h^*(\E): 0 \to h^*(B) \to h^*(A_1) \to \ldots \to h^*(A_{n})\to h^*(A) \to 0.$$  This is well-defined in  the following situations. \begin{enumerate}
    \item{The morphism $h$ is flat.}
    
    \item{Write $\mathcal E$ as the cup-product of  $1$-extensions (indexed by $i=1,\ldots,n$) \[\E_i:  0 \to E_i \to A_i \to E_{i+1} \to 0.\] Then, all $\E_i$'s are Zariski locally split  exact sequences of $\mathcal O_S$-modules. \\ If $A$, $B$ and the $A_i$'s are  vector bundles, then the $E_i$'s are vector bundles as well, so that this assumption automatically holds.}

\end{enumerate}

In practice, most base-changes will be performed in the context of item (2). In the second paper \cite{F2}, we shall work in a scheme-theoretically enhanced framework. In that upgraded setting, base-change makes sense w.r.t. arbitrary morphisms.
\subsection{Morphisms of $1$-extensions.}
Morphisms in $\EEXT_{(G, \mathcal O_S)}^1(A,B)$ are isomorphisms. Automorphisms of $1$-extensions are easily described, as follows.
\begin{lem}\label{autoext}
Let \[ \mathcal E: 0 \to B \stackrel i \to E \stackrel \pi \to A \to 0 \] be an exact sequence of $(G,\mathcal O_S)$-modules. Then, the assignment 
$$\fonctionnoname{\Hom_{(G,\mathcal O_S)} (A,B) }{\Aut_{\EEXT_{(G,\mathcal O_S)}^{1}(A,B)}(\mathcal E)}{f}{( x \in E \mapsto x+i(f(\pi(x))))}$$
is an isomorphism of abelian groups.
\end{lem}
\begin{dem}
Classical fact, working for $1$-extensions in any additive category.
\end{dem}

\subsection{Equivalence classes of Yoneda extensions.}
Let us say that two $n$-extensions $\E_1$ and $\E_2$ are linked, if there exists an $n$-extension $\E_3$, together with morphisms 
$$\xymatrix @!0 @R=2pc @C=3pc { \E_1 \ar[dr]&& \E_2 \ar[dl]\\ & \E_3. & }$$
Being linked is an equivalence relation (see \cite{O}, end of Section 2), compatible with  Baer sum, pullbacks, pushforwards and change of the base. 
\begin{defi}
 We denote by $\EXT_{(G,\mathcal O_S)}^n(A,B)$ the Abelian group of equivalence classes of linked Yoneda $n$-extensions, in the category $\EEXT_{(G,\mathcal O_S)}^n(A,B)$.
\end{defi}

\begin{lem}\label{cohhom}
	Assume that $A$ is a $G$-vector bundle on $S$. Then, there is a canonical isomorphism $$ \EXT_{(G,\mathcal O_S)}^n(A,B) \stackrel \sim \to \EXT_{(G,\mathcal O_S)}^n(\mathcal O_S,\underline \Hom(A,B)) .$$
\end{lem}

\begin{dem}
Same proof as \cite[Lemma 2.5]{DCF0}.
\end{dem}

\subsection{$G$-affine spaces.}\label{Gaffin}

In this text, we'd like to lay emphasis on the notion of an \textit{affine space}. We first define it as a set, equipped with barycentric operations, with coefficients in a commutative ring $R$. This terminology unfortunately collides with the \textit{affine $R$-scheme} $\A_R^n$, but does not lead to serious ambiguities. Note that an \textit{affine space}, whose $R$-module of translations is free of rank $n$, is isomorphic to an \textit{affine scheme} $\A_R^n$, over $R$. The empty set qualifies as an affine space, so that the intersection of affine subspaces is always an affine subspace.\\
Following tradition, we use the word \textit{torsor} under a group $M$ to denote a nonempty set $X$, equipped with a simply transitive action of $M$. A nonempty affine space is thus a torsor under the abelian group of its translations. Conversely, a torsor over an abelian group is canonically endowed with the structure of a nonempty affine space over $\Z$.\\
Details  are routine, taking into account naive actions of a given profinite group $G$, and transposing the above set-theoretic notions to algebraic geometry.

\begin{defi}
Let $M$ be a (not necessarily abelian) $G$-group. \\
A $(G,M)$-torsor is a nonempty left $G$-set $X$, equipped with a right action of $M$, subject to the following conditions :

i) The action of $M$ on $X$ is simply transitive, i.e. the arrow 
$$\fonctionnoname{X \times M }{X \times X}{(x,m) }{(x,x. m)}$$
is bijective. \\
ii) We have \[ g(x.m)=g(x).g(m),\] for all $g \in G,$ $x \in X$ and $m \in M$.
\end{defi}

Let $S=\Spec(A)$ be an affine $G$-scheme, i.e. the ring $A$ is endowed with an action of $G$.

\begin{defi}
Let $n \geq 1$ be an integer. We denote by\[ \Delta_n(A):=\{(\alpha_1, \ldots, \alpha_n) \in A^n / \sum_{i=1}^n \alpha_i =1 \}\] the usal simplex; it is a $G$-set.
\end{defi}
\begin{defi}\index{$G$-affine space!over a ring}
A $G$-affine space over $A$ is the data of a $G$-set $X$, equipped with $G$-equivariant barycentric operations, with coefficients in $A$. \\
This means that $X$ is given with $G$-equivariant functions, one for each $n \geq 2$, \[ B_n:\Delta_n(A) \times X^n \to X, \] simply denoted by \[ ((\alpha_1, \ldots, \alpha_n), (x_1, \ldots, x_n)) \mapsto \sum \alpha_i x_i,\] satisfying the usual associativity relations, together with $B_1=\Id_X$. \\ If $G$ is trivial, we simply say "affine space over $A$" for a $G$-affine space over $A$.\\
Denote by $X^G \subset X$ the subset consisting of $G$-fixed points. It is an affine space over $A^G$.

An affine map $X\stackrel f \to X'$, between $G$-affine spaces over $A$, is the data of a map \[f: X \to X',\] compatible with the barycentric operations of $X$ and $X'$. \\ The set of such morphisms, $\Hom(X,X')$,  is a $G$-affine space, in a natural way.\\
Put $$\Hom_G(X,X'):=H^0(G,\Hom(X,X')).$$ It consists of $G$-equivariant affine maps $X \to X'$, also called affine $G$-maps. \\
The collection of $G$-affine spaces over $A$ form a category, having the $G$-sets $\Hom(\cdot ,\cdot)$ as morphisms.
\end{defi}

\begin{ex}
It is clear that $(G,A)$-modules are $G$-affine spaces over $A$, in a natural way. The $G$-invariant subset $\Delta_n(A) \subset A^n$ is stable under barycentric operations in the free $G$-module $A^n$; it is thus also a $G$-affine space over $A$.
\end{ex}

\begin{exo}
Let $X$ be a $G$-affine space over $A$.

1) Show that all barycentric operations on $X$ can be recovered from the data of $$\fonction{T}{X\times X\times X}{X}{(x,y,z)}{x+y-z}$$
together with the operations $$\fonction{t_{\alpha}}{X\times X}{X}{(x,y)}{\alpha x + (1-\alpha) y},$$ for all $\alpha \in A.$

2) Assume that there exists an element $\alpha_0 \in A,$ such that $\alpha_0$ and $1-\alpha_0$ are both invertible. Show that $T$ can be recovered from the $t_\alpha$'s, for well-chosen $\alpha$'s.
\end{exo}

\begin{defi}
Let $X$ be a nonempty $G$-affine space. An affine automorphism of the shape $$\fonctionnoname{X}{X}{x}{x+y-z}$$
for some $y,z \in X$, will be called a translation, and will be denoted by `\hspace{0.1cm}$y-z$'. \\
We denote by $\overrightarrow X \subset \Aut(X)$ the (abelian) subgroup of translations. It comes naturally equipped with the structure of a $(G,A)$-module.
\end{defi}

\begin{rem}
We have $y-z=y'-z' \in \overrightarrow X$ iff $y-z+z'=y' \in X.$
\end{rem}
\begin{lem}
 Let $X$ be an nonempty $G$-affine space over $A$. Then $X$ is naturally endowed with the structure of a $(G,\overrightarrow{X})$-torsor.\\
 Conversely, let $M$ be a $(G,A)$-module, and let $X$ be a $(G,M)$-torsor. Then, $X$ is naturally endowed with the structure of a (nonempty) $G$-affine space over $A$, having $\overrightarrow{X}=M$.
\end{lem}
 \begin{dem}
 This is clear.
 \end{dem}
 \quad\\
 
The next Lemma is an adaptation of the usual construction, in classical real affine geometry, which provides a canonical embedding of an $n$-dimensional affine space, as an affine hyperplane inside an $(n+1)$-dimensional vector space.

 \begin{lem}["Modulification" of a nonempty affine space]\label{modulify}\hspace{0.1cm}\\
 Let \[ \mathcal E: 0 \to M \to N \stackrel \pi \to A \to 0\] be an exact sequence of $(G,A)$-modules. Then $$X:= \pi^{-1}(1)$$ is a nonempty $G$-affine space over $A$, with $\overrightarrow{X}=M$.\\
 Conversely, given a nonempty $G$-affine space $X$ over $A$, there exists a canonical exact sequence of $(G,A)$-modules \[ \mathcal E(X): 0 \to \overrightarrow{X} \to E(X) \stackrel \pi \to A \to 0,\] together with a canonical isomorphism of $G$-affine spaces $$X \simeq \pi^{-1}(1).$$
\end{lem}
\begin{dem}
 The first assertion is clear. The second one is less obvious. We put $$E(X):=(X \times A \times \overrightarrow X) / \sim, $$ where the equivalence relation $\sim$ is given by $$(x,\alpha, y-z) \sim (x',\alpha', y'-z')$$ if and only if $\alpha=\alpha'$ and \[ \alpha x -\alpha x' + y = y'- z' +z \in X.\]
 The class of the element $(x,\alpha, y-z)$ is then understood as "$\alpha x +y -z \in E(X)$".\\
Addition is defined by \[ (x,\alpha, y-z) + (x',\alpha', y'-z') = (x, \alpha + \alpha', \alpha' (x'-x)+ y+y'-z-z'). \]
Muliplication by scalars is given by \[ \beta.(x,\alpha, y-z):= (x, \beta \alpha, \beta(y-z)).\] The $G$-action is defined in the obvious way- as well as the extension $\mathcal E(X).$
\end{dem}

\begin{defi}[Restriction and Extension of scalars, for affine spaces]\hspace{0.1cm}\\
Let $S'=\Spec(A')$ be another affine $G$-scheme and $F:A\longrightarrow A'$ a $G$-equivariant  morphisms of rings.

i) Let $X'$ be a $G$-affine space over $A'$. We denote by $(X ')_{\vert f}$ the $G$-affine space over $A$ obtained from $X'$, using $F$ to restrict scalars.

ii) Let $X$ be a non-empty $G$-affine space over $A$. We denote by $$X \otimes_A A':=(\pi \otimes \Id_{A'})^{-1}(1)$$ the $G$-affine space over $A'$ associated to the exact sequence \[ \mathcal E(X) \otimes_A A': 0 \to \overrightarrow{X}\otimes_A A' \to E(X) \otimes_A A' \stackrel \pi \to A' \to 0.\] We thus have $\overrightarrow{X \otimes_A A'}= \overrightarrow{X} \otimes_A A'.$ If $X=\varnothing$, we set $X \otimes_A A'=\varnothing$.
\end{defi}

\begin{rem}
Extension of scalars is left adjoint to restriction of scalars, for affine maps.
\end{rem}

The previous Definitions can clearly be sheafified, in the usual fashion. We briefly explain how.

\begin{defi}
Let $S=\Spec (A)$ be an affine $G$-scheme. Let $X$ be a $G$-affine space over $A$. We denote by $\tilde X$ the $G$-sheaf on $S$ $$ U \mapsto X \otimes_A \mathcal O_S(U).$$ For each $G$-invariant open $U \subset X$, $\tilde X(U)$ is thus a $G$-affine space over $\mathcal O_S(U)$.
\end{defi}
\begin{defi}\label{WtFAFF}\index{$G$-affine space!over a scheme}
Let $S$ be a $G$-scheme. \\ A $G$-affine space over $S$ is the data of a $G$-sheaf $$\mathcal X: U \mapsto \mathcal X(U), $$with values in the category of $G$-affine spaces, such that the following holds.\\

\hspace{0.2cm}i) For all $G$-invariant open $U \subset S$, $\mathcal X(U)$ is a $G$-affine space over $\mathcal O_S(U)$.

\hspace{0.2cm}ii) For all $G$-invariant opens $V \subset U \subset S$, the morphism $$\mathcal X(\rho_{V,U}): \mathcal X(U) \to \mathcal X (V)$$ is a $G$-equivariant affine morphism, where $ \mathcal X (V)$ is considered as a $G$-affine space, via change of rings through the restriction $\rho_{V,U}: \mathcal O_S(U) \to \mathcal O_S(V) $.

\hspace{0.2cm}iii) Each $s \in S$ has an open affine $G$-invariant neighborhood $U=\Spec(A)$, such that $\mathcal X_{\vert U}$ is isomorphic to $\tilde X$, for some $G$-affine space $X$ over $A$.\\

The $G$-affine space $\mathcal X$ over $S$ is said to be everywhere nonempty, if each point $s \in S$ has a $G$-invariant open neighborhood $U$, with $\mathcal X (U) \neq \emptyset.$ In this case, there exists a unique $(G,\mathcal O_S)$-module $M$, such that $M(U)= \overrightarrow{\mathcal X(U)}$, for all $G$-invariant open subsets $U \subset S$. We denote this $M$ by $\overrightarrow{\mathcal X}$.
\end{defi}

\begin{defi}\label{defitorsG}
Let $S$ be a $G$-scheme, and let $M$ be a $(G,\mathcal O_S)$-module over $S$. A $(G,M)$-torsor (over $S$) is a $G$-affine space $\mathcal X$ over $S$, everywhere nonempty, together with an isomorphism of $(G,\mathcal O_S)$-modules $ \overrightarrow{\mathcal X} \stackrel \sim \to M.$
\end{defi}

\subsection{Twisting $1$-extensions.}

Recall that $S$ denotes  a $G$-scheme.
\begin{defi}
Let $E$ and $M$ be $(G,\mathcal O_S)$-modules. A (left) action of $M$ on $E$ is a $G$-equivariant morphism \[M \to \underline{\Aut}_{\mathcal O_S}(E) ,\] between $G$-sheaves with values in $G-\mathbf{Grp}$.
\end{defi}

\begin{ex}\label{exaact}
Let $M$ be a $(G,\mathcal O_S)$-module over $S$. Then, $M$ acts on $$E:= M \bigoplus \mathcal{O}_S,$$ by the formula (on functors of points) \[x.(y,\lambda) =(y+\lambda x,\lambda),\] for all $x,y \in M,$ and all $\lambda \in \mathcal O_S.$\\
This example deserves to be compared to  "an exponential series, truncated in degree $2$", thinking of $1+x$ as $e^x$.
\end{ex}

Let $m \geq 1$ be an integer. Let $E$ and $M$ be $(G,\mathcal O_S)$-modules over $S$. Assume given an action of $M$ on $E$. Let $P$ be a $(G,M)$-torsor over $S$. Then, one can form the twisted $(G,\mathcal O_S)$-module $E^P$, through the "usual twisting process". We briefly explain how.\\ Assume first that $S=\Spec(A)$ is affine. View $M$ and $E$ as $A$-modules, equipped with a semilinear action of $G$. We put \[ E^P:=(P \times E)/ M.\] Here, the quotient is taken with respect to the natural diagonal action of $M$, identifying $(x.m,e)$ and $(x,m.e),$ for all $e\in E$, $m \in M$ and $x \in P$. It is a set, equipped with an action of $G$, inherited from the diagonal action of $G$ on $P \times E$. Temporarily forgetting the action of $G$, it is easily shown that there is a unique structure of an $A$-module on $E^P$ such that, for any $b \in P$, the map $$\fonctionnoname{E}{E^P}{e}{\overline {(b,e)}}$$ is an isomorphism of $A$-modules.\\ We then see that the natural action of $G$ on $E^P$ occurs through semilinear automorphisms. The case $S$ arbitrary follows by gluing, using the fact that affine $G$-invariant opens of $S$ form a basis of the $G$-topology of $S$.

Twisting is functorial. More precisely, let \[f:E \to E' \] be an $M$-equivariant homomorphism between $(G,\mathcal O_S)$-modules, equipped with an action of $M$.\\ Twisting by the $(G,M)$-torsor $P$ then yields a morphism of $(G,\mathcal O_S)$-modules \[f^P:E^P \to {E'}^P. \]
The twist $E^P$ is canonically isomorphic to $E$, in each of the following cases.

\hspace{0.3cm}$i)$ The $(G,M)$-torsor $P$ is equal to $M$, the trivial torsor.

\hspace{0.3cm}$ii)$ The action of $M$ on $E$ is trivial.

We can now precisely formulate an equivalence of categories, linking $1$-extensions of $\mathcal O_S$ by $M$ to $(G,M)$-torsors. It is a sheafification of Lemma \ref{modulify}. 

\begin{prop}\label{equaffext}
Let $S$ be a $G$-scheme. Let $M$ be a $(G,\mathcal{O}_S)$-module over $S$. Let \[\mathcal E: 0 \to M \to E \stackrel \pi \to \mathcal O_S \to 0\] be an exact sequence of $(G, \mathcal O_S)$-modules. Then, the assignment \[ U \mapsto \pi^{-1}(1) \subset H^0(U,E),\] for every $G$-invariant open $U \subset S$, defines a $(G,M)$-torsor over $S$. We denote it by $X(\mathcal E)$.\\
Conversely, let $P$ be a $(G,M)$-torsor over $S$. Consider the trivial extension \[\mathcal E_0: 0 \to M \stackrel i \to M \bigoplus \mathcal O_S \stackrel \pi \to \mathcal O_S \to 0.\] Equip $M$ and $\mathcal O_S $ with the trivial action of $M$, and $M \bigoplus \mathcal O_S $ with the action of $M$ given in Example \ref{exaact}. The arrows $i$ and $\pi$ are then $M$-equivariant, and we denote by $\mathcal E(P)$ the twisted extension \[\mathcal E_0^P: 0 \to M \stackrel {i^P} \to E(P):=(M \bigoplus \mathcal O_S )^P \stackrel {\pi^P} \to \mathcal O_S \to 0.\]

The assignments \[\mathcal E \mapsto X(\mathcal E) \] and \[ P \mapsto \mathcal E(P)\] are mutually inverse equivalences of categories, from $\EExt_{(G, \mathcal O_S)}^{1}(\mathcal O_S,M)$ to the category of $(G,M)$-torsors over $S$.
\end{prop}

\begin{dem}
This is done in Lemma \ref{modulify} if $S$ is affine. The general case follows by gluing.
\end{dem}
\subsection{Representability of torsors under $G$-vector bundles.}
\begin{defi}
Let $V$ be a vector bundle over a scheme $S$. We set \[ \mathbb A(V):=\Spec(\Sym_{\mathcal O_S}(V^\vee)) \to S.\] It is the affine space associated to $V$. It represents the functor of points of $V$: for each morphism of schemes $T \to S$, we have \[ \mathbb A(V)(T)=H^0(T, V \otimes_{\mathcal O_S} \mathcal O_T). \]
\end{defi}

\begin{defi}
 Let \[\mathcal E: 0 \to V \stackrel i \to E \stackrel \pi \to \mathcal O_S \to 0\] be an extension of vector bundles, over a scheme $S$. \\ We denote its dual extension by \[\mathcal E^\vee: 0 \to \mathcal O_S \stackrel{\pi^\vee} \to E^\vee \stackrel {i^\vee} \to  V^\vee \to 0.\]  For $n \geq 1$, we then define the $n$-th symmetric power of $\mathcal E^\vee$ as   \[\Sym^{n}(\mathcal E^\vee): 0 \to \Sym^{n-1}( E^\vee ) \stackrel {\times  \pi^\vee} \to  \Sym^n( E^\vee ) \stackrel {\Sym^n(i^\vee)} \to \Sym^n( V^\vee )\to 0.\]  
\end{defi}

\begin{rem}
The  extension $\Sym^{n}(\mathcal E^\vee)$ as above, is the global version of the following local construction. For a commutative ring $A$, denote by  $A[X_0, X_1, \ldots, X_d]_{n}$ the space of polynomials with coefficients in $A$, in $d+1$ variables, homogeneous of degree $n$.   Then, we have an exact sequence of free $A$-modules \[ 0 \to A[X_0, X_1, \ldots, X_d]_{n-1} \stackrel {\times X_0} \longrightarrow  A[X_0, X_1, \ldots, X_d]_{n} \stackrel {X_0 = 0} \to  A[ X_1, \ldots, X_d]_{n} \to 0.\] Indeed, this is the particular case where $S=\Spec(A)$, $E=\mathcal O_S^{d+1}$, and $\pi$ is the projection on the first factor.
\end{rem}

The next Lemma  is a key tool in \cite{F2}.

\begin{prop}\label{GAffRep1}
Let $V$ be a $G$-vector bundle over a $G$-scheme $S$. Let $X$ be a $(G,V)$-torsor over $S$. Then, $X$ is represented by a $G$-scheme, affine over $S$.\\Slightly abusing notation, we still denote this $G$-scheme by $X \to S$.

If $X$ corresponds to an extension (of $G$-vector bundles over $S$) \[\mathcal E: 0 \to V \stackrel i \to E \stackrel \pi \to \mathcal O_S \to 0,\] then this $(G,S)$-scheme is the scheme of sections of $\pi$. \\It is an affine subspace of $\mathbb A(E )$, having $\mathbb A(V )$ as its space of translations. As such, it is the $\Spec$ of the filtered $(G,\mathcal O_S)$-Algebra \[\lim\limits_{\longrightarrow}(\Sym^{n}(E^\vee)), \] where the limit is taken with respect to the injections of the natural exact sequences \[\Sym^{n}(\mathcal E^\vee): 0 \to \Sym^{n-1}(E^\vee) \stackrel { \times \pi^\vee} \to \Sym^{n}(E^\vee) \stackrel {\Sym^{n}(i^\vee)}\to \Sym^{n+1}(V^\vee) \to 0.\] 
\end{prop}
\begin{dem}
This follow from the observation that $$X \subset \mathbb A(E )$$  is the closed subscheme given by the single affine equation \[ \pi^\vee=1.\]\end{dem}

\section{Recollections on $(G,\W_r)$ modules, $(G,\W_r)$-affine spaces and $(G,S)$-cohomology.}\label{GWtFMod}
We shall need the $G$-equivariant version of  a $\W_r$-module (see \cite{DCFL}). Teichm\"uller lifts of line bundles will also play a decisive role. For the convenience of the reader, these notions are recalled below. They give rise to a bunch of algebro-geometric structures, over  Witt vectors $\W_r$. We mainly focus in  finite depth $r$. Most structures, in depth $r=\infty$  (i.e. over $\W_{\infty}=\W$), are simply `compatible structures over $\W_r$, for all $r \geq 1$', via the following elementary construction.

For each integer $r\geq 1$, let $\mathcal S_r$ be a category, consisting of algebro-geometric structures over $\W_r$. For instance, $\mathcal S_r$ may be $(G,\W_r)$-affine spaces over a given $(G,\F_p)$-scheme $S$, or Yoneda $n$-extensions of $(G,\W_r)$-bundles over $S$. Assume given natural reduction arrows (functors) $\rho_r: \mathcal S_r \to \mathcal S_{r-1}$. This is the case in the previous examples. Then, we define a category $$\mathcal S_{\infty}=\varprojlim \mathcal S_r$$ as follows. An object of $\mathcal S_{\infty}$ is, by definition, the data of an object $X_r \in \mathcal S_r$ for all $r \geq 1$, together with compatibility isomorphisms $$\phi_r: \rho_r(X_r) \stackrel \sim \to X_{r-1},$$ for all $r \geq 2$. An arrow $$(X_r,\phi_r) \to (X'_r,\phi'_r) $$ is a collection of arrows $f_r: X_r \to X'_r$, with the obvious commutation conditions. \\ A concrete instance of this general construction appears in Definition \ref{DefiLift}, with $$\mathcal S_r=\{(G,\W_r)-\mathrm{bundles \hspace{4pt} over \hspace{4pt} }S \},$$ where $S$ is a $(G,\F_p)$-scheme $S$.

\begin{defi}[$(G,\W_r)$-module, $(G,\W_r)$-bundle, $(G,\W_r)$-affine space and $(G,M)$-torsor over $S$]\hspace{0.1cm}\\
Let $S$ be a $(G,\F_p)$-scheme. Pick $r \in  \N_{\ast} \cup \{ \infty\}$. Recall that $\W_r(S)$ is a $G$-scheme, equipped with its frobenius \[ \frob:\W_r(S) \to \W_r(S). \] A $(G,\W_r)$-module $\mathcal{M}$ over $S$ is a $\W_r(\mathcal{O}_S)$-module, equipped with a semi-linear action of $G$.\\
If $\mathcal{M}$ is locally free of finite rank as a $\W_r$-bundle, we say that $\mathcal{M}$ is a $(G,\W_r)$-bundle over $S$.\\
In case mentioning $r$ is superfluous, a $(G,\W_r)$-module over $S$ is simply referred to as a $(G, \W)$ module over $S$.

Similarly, a $(G,\W_r)$-affine space over $S$ is, by definition, a $G$-affine space over $\W_r(S)$. If $\mathcal{M}$ is a $(G,\W_r)$-module over $S$, a $(G,\mathcal{M})$-torsor is defined as in \ref{defitorsG}, where $\mathcal{M}$ is viewed as a $(G,\mathcal O_{\W_r(S)})$-module.
\end{defi}

If $\mathcal{M}$ is a $(G,\W_r)$-module over $S$, Proposition \ref{equaffext} implies that the category of $(G,\mathcal{M})$-torsors is equivalent to the category $$\EExt_{(G, \W_r(\mathcal O_S))}^{1}(\W_r(\mathcal O_S),\mathcal{M}).$$

\begin{defi}($(G,S)$-cohomology)\\
Let $S$ be a $(G,\F_p)$-scheme, and let $r \in  \N_{\ast} \cup \{ \infty\}$.\\ Let $\mathcal M$ be a $(G,\W_r)$-module over $S$.\\
For $n\geq 0$, we set $$H^n((G,S),\mathcal M):=\EXT_{(G, \W_r(\mathcal O_S))}^{n}(\W_r(\mathcal O_S),\mathcal{M}).$$ In particular, $H^1((G,S),\mathcal M)$ is the abelian group formed by isomorphism classes of $(G,\mathcal{M})$-torsors over $S$.\\
\end{defi}

\begin{defi}[Lifting  $(G,\W_r)$-bundles]\hspace{0.1cm} \label{DefiLift}\\
Let $\mathcal{M}_r$ be a $(G,\W_r)$-bundle over $S$. Pick an integer $s \geq r$.\\A lifting of  $\mathcal{M}_r$  to a $(G,\W_s)$-bundle, is   a $(G,\W_s)$-bundle $\mathcal{M}_s$ over $S$, together with an isomorphism  of $(G,\W_r)$-bundles $$\mathcal{M}_s \otimes_{\W_s}\W_r \stackrel \sim \to \mathcal{M}_r.$$
If  specifying an isomorphism is not necessary, we simply say that  $\mathcal{M}_r$ lifts to  $\mathcal{M}_s$.\\
We say that $\mathcal{M}_r$ lifts completely if $\mathcal{M}_r$ admits a compatible system of liftings, i.e. for every $s>r$, a $(G,\W_s)$-bundle $\mathcal{M}_s$ is given,  together with isomorphisms $$\mathcal{M}_{s+1}\otimes_{\W_{s+1}}\W_s \stackrel \sim \to  \mathcal{M}_{s}.$$
\end{defi}

\subsection{Teichmüller lift  of  a line bundle.}

The multiplicative section for Witt vectors provides a compatible system of liftings, for $G$-line bundles over $S$. Its main properties are gathered in the next Proposition, taken from \cite[§4]{DCFL}.

\begin{prop}\label{TeichWitt}
Let $S$ be a $(G,\F_p)$-scheme. Let $L$ be a $G$-line bundle over $S$.\\
For any $r \geq 1$, there exists a canonical lift of $L$ to a $(G,\W_r)$-line bundle over $S$.

It is the $r$-th Teichm\"uller lift of $L$, denoted by $\W_r(L)$. Teichm\"uller lifts of $L$ are compatible, in the following sense.

1) We have $\W_1(L)=L$.\\
2) For all $s \geq r \geq 1$, we have a natural exact sequence (of $G$-$\W$ modules over $S$) \[ 0 \to (\frob^r)_*(\W_{s-r}(L^{\otimes p^r})) \to \W_{s}(L) \stackrel {\pi_{s,r,L}} \to \W_{r}(L) \to 0.\]

Furthermore, the surjection $\pi_{s,r,L}$ admits a canonical (non-linear, sheaf-theoretic, $G$-equivariant) section- its Teichm\"uller section. We denote it by $\tau_{s,r,L}$, or simply by $\tau_L$. It is obtained by twisting the usual Teichm\"uller section, by the $\G_m$-torsor associated to $L$.
\end{prop}

\section{Cyclotomic pairs and smooth profinite groups.}\label{CyclotomicSmooth}

\subsection{$(n,e)$-cyclotomic pairs.}
Set $$\Z/p^{\infty}\Z:=\Z_p.$$ We endow $\Z_p$-modules of finite-type  with the $p$-adic topology.

\begin{defi}
Let $G$ be a profinite group. Let $e\in \mathbb{N}^{\ast}\cup\{ \infty\}$ be a number.\\
A $(\Z/p^e\Z,G)$-module $\mathcal{M}$ is a $\Z/p^e\Z$-module of finite type, endowed with a continuous action of $G$. (In case $e < \infty$, the action is thus naive.)
\end{defi}
 For an integer $1\leq f\leq e$ and a $(\Z/p^e\Z,G)$-module $\mathcal{M}$, we put $$ \mathcal{M}/p^f := \mathcal{M} \otimes_{\Z_p}(\Z/p^f\Z),$$ and we denote the quotient map by $$\pi_{e,f}:\mathcal{M}\to \mathcal{M}/p^f.$$ 
\begin{defi}\label{cyclotomicpair}
Let $n\geq 1$ and $e\in \mathbb{N}^{\ast}\cup\{ \infty\}$. Let $\mathcal{T}$ be a $(\Z/p^{e+1}\Z,G)$-module, free of rank one as a $\Z/p^{e+1}\Z$-module. We say that the pair $(G,\mathcal{T})$ is $(n,e)$-cyclotomic if, for every open subgroup $H\in G$, the morphism $$H^n(H,\mathcal{T}^{\otimes n}) \to H^n(H,(\mathcal{T}/p)^{\otimes n}),$$
induced by $\pi_{e+1,1}$, is surjective. The integer $e$ is then called the depth of the cyclotomic pair.
\end{defi}

\begin{rem}
By a limit argument, "open" may be replaced by "closed" in the preceding definition.
\end{rem}

\begin{rem}
Let $\mathcal{T}$ be a $(\Z/p^{e+1}\Z,G)$-module, free of rank one as a $\Z/p^{e+1}\Z$-module. Let $G_1 \subset G$ be an open subgroup of prime-to-$p$ index. Then, the pair $(G,\mathcal{T})$ is $(n,e)$-cyclotomic if, and only if, the pair $(G_1,\mathcal{T})$ is $(n,e)$-cyclotomic, by a usual restriction/corestriction argument. \\
In particular, we may take $G_1$ to be the kernel of the multiplicative character $$\chi_1: G \to \F_p ^\times,$$ giving the action of $G$ on $\mathcal{T}/p$. By doing so, many problems can be reduced to the case where $\mathcal{T}/p \simeq \F_p$ is endowed with the trivial action of $G$.
\end{rem}
\begin{rem}
Let $(G,\mathcal{T})$ be an $(n,e)$-cyclotomic pair. Then, for every integer $f$ such that $1 \leq f <e+1$ and for every open subgroup $H\in G$, the arrow $$H^n(H,\mathcal{T}^{\otimes n}) \to H^n(H,(\mathcal{T}/p^f)^{\otimes n})$$ is surjective. The proof is by induction on $f$, using the exact sequences \[ 0 \to \mathcal{T}/p^f \stackrel {\times p} \to \mathcal{T}/p^{f+1} \to \mathcal{T}/p\to 0. \]
\end{rem}

If $(G,\mathcal{T})$ is a $(n,e)$-cyclotomic pair, then $\mathcal{T}$ is given by a continuous character $$\chi: G \to (\Z/p^{e+1}\Z)^\times,$$ 
 analoguous to the  cyclotomic character in number theory. For $i\geq 1$, set $$\Z/p^{e+1}\Z\hspace{0.05cm}(i):= \mathcal{T}^{\otimes^i_{\Z_p}},$$ and for any $\Z/p^{e+1}\Z$-module $\mathcal{M}$, we put $$\mathcal{M}(i):=\mathcal{M}\otimes_{\Z_p} \Z/p^{e+1}\Z\hspace{0.05cm}(i).$$
These are called `cyclotomic twists'.
\begin{ex}
Let $F$ be a field of characteristic not $p$. Let $G=Gal(F_{sep}/F)$ be the Galois group of a separable closure $F_s/F$. Let $$ \mu:=\varprojlim_r \mu_{p^r}$$ be the Tate module of roots of unity of $p$-primary order. It is a free $\Z_p$-module of rank one, endowed with a continuous action of $G$. Kummer theory states that the pair $(G,\mu)$ is $(1,\infty)$-cyclotomic. As explained in the introduction, the statement of the Bloch-Kato-Milnor conjecture is equivalent to  $(G,\mu)$ being $(n,1)$-cyclotomic, for every $n \geq 1$. Other fundamental examples of cyclotomic pairs are given in \cite{DCF0}.
\end{ex}

We conclude this section with an instructive  exercise.

\begin{exo}\label{exopurep}
Assume that $p=2$. The goal is to present a group-theoretic version of the famous identity \[ (x) \cup (x)=(-1) \cup (x) \in H^2(F, \Z/2),\] valid for every $x \in F^\times$, with $F$ a field of characteristic not $2$.

\begin{enumerate}
    \item Let $G$ be a profinite group. Let $$\chi: G \to \{1,-1\} (\simeq \F_2)$$ be a character of $G$. Set $\Z/4 (\chi):=\Z/4$, on which $G$ act via $\chi$. Let\[\mathcal E_1: 0 \to \Z/2 \to  E_1  \to \Z/2 \to 0\] be an extension of $(\F_2,G)$-modules, with class $e_1 \in H^1(G,\F_2)$.\\ Assume that $\mathcal E_1$ lifts to an extension of $(\Z/4,G)$-modules \[\mathcal E_2: 0 \to \Z/4 (\chi) \to  E_2  \to \Z/4 \to 0. \]  Show that  the identity \[ e_1 \cup e_1=\chi \cup e_1 \in  H^2(G,\F_2)\] holds.\\
\item Show that the pair $(G,\Z/4(\chi))$ is $(1,1)$-cyclotomic, iff the identity $$ c \cup c=\chi \cup c\in H^2(H,\Z/2)$$ holds, for all open subgroups $H \subset G$, and for all $c\in H^1(H,\Z/2)$.\\
\item  Adapt  the statement above, for odd $p$. This leads to an equivalent definition of $(1,1)$-cyclotomic pairs, using only mod $p$ cohomology.
\end{enumerate}
\end{exo}

\subsection{$(n,e)$-smooth profinite groups.}

We can now  state the definition of smooth profinite groups. Equivalent definitions are provided in the Appendix.

\begin{defi}[Smooth profinite group]\label{DefiSmooth}\quad \\
Let $n\geq 1$ and $e\in \mathbb{N}^{\ast}\cup \{\infty\}$. A profinite group $G$ is said to be $(n,e)$-smooth if the following lifting property holds. \\
Let $A$ be a perfect $\F_p$-algebra equipped with a naive action of $G$. Let $L_1$ be a locally free $A$-module of rank one, equipped with a semi-linear naive action of $G$. Let $$c \in H^n(G,L_1)$$ be a cohomology class. Then, there exists a lift of $L_1$, to a $(\W_{e+1}(A),G)$-module $L_{e+1}[c]$, locally free of rank one as a $\W_{e+1}(A)$-module and depending on $c$, such that $c$ belongs to the image of the natural map \[H^n(G,L_{e+1}[c]) \to H^n(G,L_1). \]
\end{defi}

\section{Lifting $(G,\W_r(L)(1))$-torsors.}\label{SecLift}

\subsection{Why cohomology with $\W_r(L)$-coefficients?\\}
Let $X$ be a  variety over a field $F$, of characteristic not $p$. Denote by $G=\pi_1(X)$ `the' \'etale fundamental group of $X$, and by $\Z_p(1)$ the usual Tate module.\\ One can think of the groups $$H^i((G,S),\W_r(L)(j)),$$ for various $(G,\F_p)$-schemes $S$ and $G$-line bundles $L$ over them, as $$`H^i_{et}(X,\W_r(L)(j))\mbox{'}$$ where  $L$ is a system of mod $p$ coefficients of  multiplicative nature,  extending the notion of rank one $(\Z/p^r)$-local system on $X$. This analogy, that  solely involves $G$, is  accurate if $X$ is a $K(\pi,1)$. Our point of view is to perform geometric operations on the coefficients of the cohomology, instead of  the variety $X$ itself. 

\subsection{Lifting geometrically split extensions.}

In this section, $n$ is a positive integer, $e \in \N^{*} \cup \{\infty \}$ and $S$ denotes a $(G,\F_p)$-scheme. We assume that $$(G,\Z/p^{1+e}(1))$$ is a $(n,e)$-cyclotomic pair.

The next Definitions are a prerequisite for stating the main Theorems of this section.  They are especially meaningful, when $n=1$.

\begin{defi}[Lifting cohomology]
Let $S$ be a $(G,\F_p)$-scheme and $L$ be a $G$-linearized line bundle over $S$.\\ Let $1\leq r \leq e$ be an integer and $$c_r \in H^n((G,S), \W_{r}(L)(n)) $$ be a cohomology class.

If $s \in \{r+1,...,e+1\}$ is an integer and $$c_s \in H^n((G,S),\W_{s}(L)(n)) $$ is a cohomology class, we say that $c_s$ lifts $c_r$, if $c_s$ is sent to $c_r$ by the map \[H^n((G,S), \W_{s}(L)(n)) \to H^n((G,S), \W_{r}(L)(n)) \] induced by the natural reduction arrow \[\W_{s}(L)(n) \to \W_{r}(L)(n), \] between $G$-$\W$ modules on $S$.

Accordingly, we say that a $(G,\W_{r}(L)(1))$-torsor lifts, if its cohomology class does.
\end{defi}

\begin{defi}[Strongly geometrically trivial classes]\label{DefiGeomTriv}
Let $S$ be a $(G,\F_p)$-scheme,  and let
$$\mathcal{E}:0\to M_0\to M_1 \overset{\pi}{\to}  M_2\to 0$$
be a short exact sequence of $(G,\W_r)$-modules on $S$.\\
We say that $\mathcal{E}$ is  geometrically trivial (or geometrically split), if $\pi$ admits an  $\mathcal O_S$-linear (non-necessarily $G$-equivariant) section.

More generally, an  $n$-extension of $(G,\W_r)$-modules over $S$
$$0\to M_0\overset{f_0}{\to} M_1 \overset{f_1}{\to} ...\overset{f_{n-1}}{\to} M_{n}\overset{f_{n}}{\to}  M_{n+1}\to 0$$
is strongly geometrically trivial if the following holds. Split off the extension, as the concatenation of short exact sequences \[ \mathcal{E}_i: 0 \to A_{i-1} \to M_i \to A_i \to 0,\] $i=1,\ldots, n$, given by the kernels and cokernels of the $f_i's$. Then, all the $\mathcal E_i$'s are  geometrically trivial.

Accordingly, for a   $(G,\W_r)$-module $M$, we say that a $(G,M)$-torsor, or more generally a cohomology class $c\in H^n((G,S),M)$, is (strongly) geometrically trivial if it can be represented by a (strongly) geometrically trivial $n$-extension of $(G,\W_r)$-modules $$0\to M=M_0\overset{f_0}{\to} M_1 \overset{f_1}{\to} ...\overset{f_{n-1}}{\to} M_{n}\overset{f_{n}}{\to}  \W_r(\mathcal O_S)\to 0.$$ Strongly geometrically trivial classes form a subgroup \[ H^n_{sgt}((G,S),M) \subset H^n((G,S),M).\]
\end{defi}

\begin{rem}
For $n=1$, we have \[ H^1_{sgt}((G,S),M)=\Ker(H^1((G,S),M) \to H^1(S,M)). \] For $n \geq 2$, we have an inclusion \[ H^n_{sgt}((G,S),M) \subset \Ker(H^n((G,S),M) \to H^n(S,M)), \]  which is, in general,  far from being an equality.
\end{rem}

\begin{lem}\label{AffGeomTriv}
Let $S$ be an affine $(G,\F_p)$-scheme. Let $M$ be  a   $(G,\W_r)$-module on $S$. Then, all cohomology classes are strongly geometrically trivial: we have \[ H^n_{sgt}((G,S),M)=H^n((G,S),M) . \]
\end{lem}
\begin{dem}
Adapting the process of  \cite[Lemma 5.1]{DF}, we can represent a given cohomology class  $c \in H^n((G,S),M) $ by an $n$-extension of $(G,\W_r)$-modules on $S$ $$\mathcal C: 0\to M {\to} M_1 \overset{f_1}{\to} ...\overset{f_{n-1}}{\to} M_{n}\overset{f_{n}}{\to} \W_r(\mathcal O_S) \to 0,$$ where $M_2,\ldots, M_n$ are $(G,\W_r)$-bundles. Such an extension is strongly geometrically split. Indeed, over an affine base, short exact sequences of quasi-coherent modules, having a vector bundle as cokernel, are split.
\end{dem}

\begin{prop}\label{BuildGeomTriv}
Let $S$ be a $(G,\F_p)$-scheme.  Let $M$ be  a   $(G,\W_r)$-bundle on $S$. For all $n \geq 1$, there is a natural surjection \[\gamma: H^n(G,H^0(S,M)) \to H^n_{sgt}((G,S),M).\]
\end{prop}
\begin{dem}
One reduces w.l.o.g. to the case of a finite group $G$.
Start with a class $c \in H^n(G,H^0(S,M)) $. Using  \cite[Lemma 5.1]{DF}, represent it by an $n$-extension of $\Z/p^r[G]$-modules $$\mathcal C: 0\to H^0(S,M){\to} E_1 \overset{g_1}{\to} ...\overset{g_{n-1}}{\to} E_{n}\overset{g_{n}}{\to}  \Z/p^r \to 0,$$  where     $E_2, E_3, \ldots, E_n$ are finite free $\Z/p^r$-modules. It is then straighforward to check, by descending induction on $i$, that $\Ker(g_i)$ is also  a finite free $\Z/p^r$-module, for $i=2,\ldots,n$. Applying \hspace{0.1cm}$\cdot\otimes_{\Z/p^r} \W_r(\mathcal O_S)$ to $\mathcal C$ thus preserves its exactness, yielding an  $n$-extension of $(G, \W_r)$-modules on $S$ $$ \mathcal C_S: 0\to H^0(S,M) \otimes \W_r(\mathcal O_S) {\to} M_1 \overset{f_1}{\to} ...\overset{f_{n-1}}{\to} M_{n}\overset{f_{n}}{\to} \W_r(\mathcal O_S)\to 0,$$ where $M_i:=E_i \otimes \W_r(\mathcal O_S)$. Note that $M_2,\ldots, M_n$ are $(G,\W_r)$-bundles.\\Denote by $$ \alpha: H^0(S,M) \otimes \W_r(\mathcal O_S) \to M$$ the canonical arrow, given by restricting global sections. Form the pushforward  $$\mathcal E:=\alpha_*(\mathcal C_S): 0 \to M  {\to} M'_1 \overset{f_1}{\to} M_2\overset{f_2}{\to}  ...\overset{f_{n-1}}{\to} M_{n}\overset{f_{n}}{\to} \W_r(\mathcal O_S)\to 0;$$ it is an   $n$-extension of $(G, \W_r)$-modules on $S$ whose class in $H^n_{sgt}((G,S),M)$ is denoted by $\gamma(c)$. This construction defines the arrow $\gamma$.\\
To show that $\gamma$ is surjective, start with  $e \in H^n_{sgt}((G,S),M)$, represented by a strongly geometrically trivial $n$-extension of $G \W_r$-modules on $S$ $$\mathcal E: 0 \to M  {\to} M_1 \overset{f_1}{\to} M_2\overset{f_2}{\to}  ...\overset{f_{n-1}}{\to} M_{n}\overset{f_{n}}{\to} \W_r(\mathcal O_S)\to 0.$$ Taking global sections then yields  an $n$-extension of $\Z/p^r[G]$-modules \vspace{0.2cm}\\
{\centering\noindent\makebox[351pt]{$H^0(S,\mathcal E): 0 \to H^0(S,M) {\to} H^0(S,M_1) \overset{g_1}{\to}  ...\overset{g_{n-1}}{\to} H^0(S,M_{n})\overset{g_{n}}{\to} H^0(S,\W_r(\mathcal O_S))\to 0,$}}

which we pullback by the arrow 
$$\fonctionnoname{\Z/p^r}{H^0(S,\W_r(\mathcal O_S))}{1}{1}$$
to get an $n$-extension of $\Z/p^r[G]$-modules $$\mathcal C: 0 \to H^0(S,M) {\to} \hspace{0.1cm} \cdots \hspace{0.1cm} {\to} \Z/p^r\to 0.$$  One then checks that  $\gamma$  maps $[\mathcal C] \in  H^n(G,H^0(S,M))$ to $e$. \end{dem}
\begin{rem}
In Definition \ref{DefiGeomTriv}, assume that $M$ is a $(G,\W_r)$-bundle. Using (the proof of) Proposition \ref{BuildGeomTriv},  every element of $ H^n_{sgt}((G,S),M)$ can be represented by a strongly geometrically trivial $n$-extension of $(G,\W_r)$-bundles $$0\to M=M_0\overset{f_0}{\to} M_1 \overset{f_1}{\to} ...\overset{f_{n-1}}{\to} M_{n}\overset{f_{n}}{\to}  \W_r(\mathcal O_S)\to 0.$$
\end{rem}
\section{Statement of Theorem A.}\label{WeakLiftSec}
In this section, we state and prove the main result of this article: a generalization of classical Kummer theory, for $H^1(\Gal(F_{sep}/F),\mu_{p^r})$, to the broader context of torsors for $(G,\W_r)$-line bundles, over a $(G,\F_p)$-scheme $S$.

\begin{thmA*} \label{WeakLift}
Let $(G,\Z/p^{1+e}(1))$ be a $(n,e)$-cyclotomic pair, relatively to some integer $n \in \N^{\ast}$ and $e\in \N^{\ast} \cup \{\infty\}$.

 Pick an integer $1 \leq r \leq e$. Let $S$ be a $(G,\F_p)$-scheme and $L$ be a $G$-linearized line bundle over $S$. Consider a strongly geometrically trivial class $$c_r \in H_{sgt}^n((G,S),\W_r(L)(n)).$$
 
 Then, there is an integer $m \geq 0$ such that the frobenius pullback $c_r^{(m)}$ of $c_r$ lifts to a  strongly geometrically trivial class, via \[ H_{sgt}^n((G,S),\W_{1+e}(L^{(m)})(n)) \to H_{sgt}^n((G,S),\W_r(L^{(m)})(n)). \]
 In particular, if $S$ is a perfect affine scheme, the natural arrow \[ H^n((G,S),\W_{1+e}(L)(n)) \to H^n((G,S),\W_r(L)(n)) \] is onto. Consequently, $G$ is a $(n,e)$-smooth profinite group.
\end{thmA*}

\begin{rem}
By the very definition of a cyclotomic pair, Theorem A also clearly holds if we replace $G$ by an open (or even closed) subgroup $H \subset G$. Its proof actually invokes a tremendous amount of such subgroups.
\end{rem}
\begin{rem}
For proving Theorem A, without loss of generality, we can assume that $\F_p(1)\simeq \F_p$ has the trivial $G$-action. Indeed, the action of $G$ on $\F_p(1)$ occurs through a multiplicative character $$\xi:G \to \F_p^\times$$ whose kernel $G_0$ has index dividing $p-1$, hence prime-to-$p$. Invoking the usual restriction-corestriction argument, it is then free to replace $G$ by $G_0$.
\end{rem}

\section{Factorizing frobenius.}

In this section, we will encounter infinite dimensional $\mathbb{F}_p$-vector spaces, endowed with a naive action of $G$-- for instance, $(\F_p,G)$-algebras, which are of finite-type as $\F_p$-algebras. We thus state the following definition.
\begin{defi}\hfill\\
If $G$ is a profinite group, an $[\mathbb{F}_p,G]$-module is an $\F_p$-vector space, equipped with a naive $\F_p$-linear action of $G$.
\end{defi}

\begin{rem}
 For $G$ finite, an $[\mathbb{F}_p,G]$-module is simply a module over the group algebra $\F_p[G]$.
\end{rem}

\begin{defi}[Permutation modules]\label{defipermmodule}
An $[\F_p,G]$-module is said to be a permutation module if it has an $\F_p$-basis (possibly infinite) which is permuted by $G$.\\In other words, $P$ is permutation, if it is isomorphic to an $[\F_p,G]$-module of the shape $\F_p^{(X)}$, where $X$ is a  $G$-set (the action of $G$ being naive).

We say that a morphism of $[\F_p,G]$-modules $$f:M \to N$$ factors through a permutation module if there is a permutation module $P$ and a factorization 
$$\xymatrix @C=2pc @R=1pc {
& P \ar[rd]^{g_2}& \\
M \ar[ru]^{g_1} \ar[rr]_f & & N}$$
Such morphisms form a subgroup of $\Hom_{[\F_p,G]}(M,N)$.
\end{defi}
The goal of this section is to provide Theorem \ref{FIT}, a remarkable algebraic device (and a key ingredient in the proof of Theorem A). 

\begin{lem}\label{lemcomalg}
 Let $A$ be an $(\F_p,G)$-algebra, reduced and of finite-type as an $\F_p$-algebra. Setting  $B:=A^G$, the following assertions hold.
\begin{itemize}
\item[\textit{i)}] The $\F_p$-algebra $B$ is of finite-type, and $A$ is finite, as a $B$-module.
\item[\textit{ii)}] There exists a finite $G$-set $X$, and an element $f \in B,$ which is not a zero divisor in $A$, with the following properties:
\begin{itemize}
\item[\textit{a)}] The algebra $A_f / B_f$ is finite \'etale.
\item[\textit{b)}] There exists $G$-equivariant homomorphisms of $B$-modules
\[\phi: A\to B^X,\qquad\text{and}\qquad\psi: B^X \to A,\]
such that
\[\psi \circ \phi = f \Id.\]
\end{itemize}
\item[\textit{iii)}] The extension of $[\F_p,G]$-modules \[(\mathcal E_1): 0 \to A \stackrel {\times f} \to A \stackrel \pi \to A/f \to 0 \] is split by pullback by the natural quotient map $q: A/f^2 \to A/f.$ 

\end{itemize}

\end{lem}

\begin{proof}
Point i) is classical; let us prove ii). Denote by $H \subset G$ the kernel of the action of $G$ on $A$; it is an open subgroup.

Assume first that $A$ is a domain. Denote by $L$ (resp. $K$) the field of fractions of $A$ (resp. of $B$). By Artin's Lemma, the extension $L/K$ is Galois, with Galois group $G/H$. Put $ X:=G/H$. Then, by the normal basis theorem, there exists a $G$-equivariant isomorphism of $K$-vector spaces $L \stackrel \sim \to K^X.$ The existence of $f \in B$, enjoying the properties required in $a)$ and $b)$, readily follows. This argument instantly extends to the case where $A$ is a finite product of domains, after noting that the group $G$ naturally permutes the factors of the finite product in question (which correspond to the primitive idempotents of $A$).

Let us deal now with the general case: denote by $P_1, \ldots, P_s$ the generic points of $\Spec(A).$ Put $$K_i:= A_{P_i};$$ it is a reduced Artinian ring, hence a field. The canonical map $$ \iota: A \to \prod_{i=1}^s K_i$$ is injective. \\ 
For each index $i=1, \ldots ,s,$ there exists an element $$a_i \in (\cap_{ j \neq i} P_j)-P_i.$$ Equivalently, the element $a_i$ is nonzero in $K_i$, but vanishes in all $K_j$'s, for $j \neq i$. Put $$a:=a_1+ \ldots + a_s.$$ We then have $$ a_i^2- a a_i=0 \in A$$ for all $i$; indeed, these elements vanish in all $K_j$'s. The element $a \in A$ is not a zero divisor, hence so is $$b := N_{G/H}(a) \left(=\prod_{g \in G/H} g\cdot a \right) \in B.$$ Furthermore, the elements $$e_i:= \frac {a_i} {a} \in A_b$$ are primitive idempotents, decomposing $A_b$ into a finite product of domains. We are thus reduced to the previous case.

To prove $iii)$, consider first the commutative diagram of $[\F_p,G]$-modules \[ \xymatrix{(\mathcal E_2):0 \ar[r] & A \ar[r]^{\times f^2} \ar[d]^\phi& A \ar[r] \ar[d]^{\phi}& A/f^2\ar[r] \ar[d]^{\phi/f^2}& 0 \\(\mathcal F_2): 0 \ar[r] & B^X \ar[r]^{\times f^2} \ar[d]^{\psi}& B^X \ar[r] \ar[d]^{\psi}& (B/f^2)^X \ar[r] \ar[d]^{\psi /f^2}& 0 \\ (\mathcal E_2): 0 \ar[r] & A \ar[r]^{\times f^2} & A \ar[r] & A/f^2\ar[r] & 0.} \] The middle exact sequence $\mathcal F_2$ is split, since $B \to B/f^2$ splits as an $\F_p$-linear map. Since $\psi \circ \phi=f \Id$, it follows that $$f \mathcal E_2=0 \in \Ext^1_{[\F_p,G]}(A/f^2,A).$$ The diagram \[ \xymatrix{(\mathcal E_1):0 \ar[r] & A \ar[r]^{\times f} \ar@{=}[d]& A \ar[r] \ar[d]^{\times f}& A/f\ar[r] \ar[d]^{\times f}& 0 \\ (\mathcal E_2): 0 \ar[r] & A \ar[r]^{\times f^2} & A \ar[r] & A/f^2\ar[r] & 0} \] shows that $q^*(\mathcal E_1)=f \mathcal E_2$. This completes the proof.
\end{proof}

Arguably, the next theorem maximizes the product (simplicity $\times$ depth), among all  results of this article.

\begin{thm}[]\label{FIT}

Let $A$ be an $(\F_p,G)$-algebra, of finite-type as an $\F_p$-algebra. Then, there exists an integer $m\geq 0$ such that, as a morphism of $[\F_p,G]$-modules, 
$$\frob_A^m:A \to A$$ factors through a permutation module.
\end{thm}
\begin{proof}
Let $i \geq 0$ be such that the nilradical $\mathcal N$ of $A$ satisfies $\mathcal N^{p^i}=0$. Then $\frob_A^i$ canonically factors through $A \to A_{red}$ . We can thus assume that $A$ is reduced, and proceed by induction on the (Krull) dimension of $A$. We use the notation and the results of Lemma \ref{lemcomalg}. By induction, there exists an integer $m' \geq 0$, working for $A/f$. By point $iii)$ of Lemma \ref{lemcomalg}, there exist a morphism of $[\F_p,G]$-modules $s: A/f^2 \to A$, such that $ \pi \circ s=q$. Denote by $\phi: A/f \to A/f^2$ the canonical map, sending $a$ (mod $f$) to $a^p$ (mod $f^2$). Put $$F_1:= s \circ \phi \circ \frob_{A/f}^{m'} \circ \pi: A \to A;$$ it is a morphism of $[\F_p,G]$-modules, factoring through a permutation module (because $\frob_{A/f}^{m'}$ does). Then, the difference $ \frob^{m'+1}-F_1$ takes values in the ideal $fA \subset A$. Hence, there exists a morphism of $[\F_p,G]$-modules $$F_2: A \to A,$$ such that \[\frob_A^{m'+1}=F_1+fF_2.\] By point $ii)$ of Lemma \ref{lemcomalg}, the multiplication by $f$ morphism $A \to A$ factors through a permutation module- hence so does $fF_2$. In particular, $m:=m'+1$ does the job.
\end{proof}

\begin{exo}\label{exoFIT}
Refining the proof of Theorem \ref{FIT},  show the following more precise statement, under the same assumptions. Consider the product $$\mathbf P(A):=\prod_{x \in \mathrm{Max}(A)} k(x), $$ taken over all closed points $x \in Spec(A)$, with residue field the finite field $k(x)$. Show that it is a permutation $[\F_p,G]$-module, and that there exists an integer $m \geq 0$, such that $$ \fonction{\frob^m_A}{A}{A}{a}{a^{p^m}}$$factors through the natural map $A\to \mathbf P(A),$ as a morphism of $[\F_p,G]$-modules.
\end{exo}

\begin{qu}(Does Theorem \ref{FIT} hold for modules?)\\
Let $M$ be an $A[G]$-module, which is finite locally free an an $A$-module. Is there an integer $m \geq 0$ such that $$ \fonction{\frob^m_M}{M}{M^{(m)}}{x}{1 \otimes x}$$ factors through the natural map $M\to \mathbf P(A) \otimes_A M,$ as a morphism of $[\F_p,G]$-modules? In general, the answer is most likely ``no''.
\end{qu}

\section{The transfer, for finite vector spaces.}\label{sectransfer}

 In our unpublished work \cite{DCF0}, we proposed a systematic study of finite modules over $\W_r(k)$, $k$ a finite field. There, a noteworthy algebraic construction is the so-called \textit{transfer}. At a minor cost, it actually implies a refined explicit version of Theorem \ref{FIT}- in the spirit of Exercise \ref{exoFIT}. This section is a synthetic presentation of the transfer. In this work, it is only used to prove Proposition \ref{propnoperfcov}, which is not needed to prove the Smoothness Theorem.\\
 
In this section, $q=p^r$ is a power of $p$, $k\simeq \F_q$ is a field with $q$ elements, and  $V$ is a finite-dimensional $k$-vector space, of dimension $d$.\\
Consider the frobenius 
$$\fonction{ \frob^r}{ \A_k( V)}{\A_k ( {V^{(r)}})}{v}{ v^{(r)}:=v \otimes 1}$$
as a morphism of $k$-varieties, or equivalently here, as a \textit{polynomial law} (see \cite{Fe}). Up to the choice of a basis, it is given by raising all coordinates to the $q$-th power. Observe that the frobenius \textit{map}   $$ \frob^r: V \to {V^{(r)}}$$ (the frobenius morphism applied to $k$-rational points) is a $k$-linear isomorphism, implicitly used in  this section. 

\begin{nota}
	The frobenius law $$\frob^r:  \A_k( V) \to  \A_k(V)$$  is  denoted by $F_V$, or just  by $F$ if the dependence in $V$ is clear. 
\end{nota}

\begin{defi}\label{defiTHV}
    Let $H \subset V$ be a hyperplane. Let $\pi \in V^\vee$ be a linear form with kernel $H$.
The  formula
$$\fonction{T_{H,V}}{\A_{k}(V)}{\A_{k}( H)}{v}{F(v) -\pi(v)^{q-1} v}$$
defines a polynomial law,  homogeneous of degree $q$.\\
Let us briefly justify that this formula makes sense: since $\lambda^{q-1}=1$, for all $\lambda \in k^\times$, it is clear that $ T_{H,V}$  depends  on $H$ only, not on the choice of $\pi$.  Moreover, the computation
 $$\pi(T_{H,V}(v)) =   \pi(F(v)) -\pi(v)^{q}= \pi(v)^{q} - \pi(v)^{q}=0,$$ 
shows that $ T_{H,V}$ actually takes its values in $ \A_{k}( H )\subset \A_{k}(V)$.\\
The arrow $T_{H,V}$ is called the transfer, from $V$ to $H$. By universal property of divided powers, it is given by a $k$-linear map \[\Gamma^q_k(V) \to H,\] which we also denote by $T_{H,V}$.
\end{defi}
\begin{lem}\label{lemTFH}

Consider the composite $$ \Phi_1 :\A_k(V)   \xrightarrow{T_V} \A_k(\bigoplus_{H} H)\to  \A_k(V),$$ where the sum is taken over all $k$-hyperplanes $H \subset V$, where $T_V$ is the sum of all $T_{H,V}$'s,  and the second arrow is  the sum of the inclusions $H \to V$. \\ Then $\Phi_1=F_V.$
\end{lem}

\begin{dem}
	We argue on the level of the functor of points of  the $k$-variety $ \A_k(V) $.\\
	Let $k'/k$ be a commutative $k$-algebra. Pick  $v \in V \otimes_k k'$.\\
	Since each hyperplane in $V$ is the kernel of $q-1$ linear forms  (a number which equals $-1$ modulo $p$), the composite under consideration sends $v$ to $$-\sum_{\pi \in V^*, \pi \neq 0}(F(v)-\pi(v)^{q-1}v).$$ The Proposition then follows from the fact that $\vert V^* \vert =-1$ mod $p$, and that $$\sum_{\pi \in V^* }\pi(v)^{q-1} =0.$$ Indeed,  this sum  is of the shape $$\sum_{x \in k^d} P(x),$$ where $P \in k'[X_1, \ldots, X_d]$ is a homogeneous polynomial, of degree $q-1$. It is a classical fact (used in the proof of the Chevalley-Warning Theorem) that the only monomials which can contribute to this sum are those of the form $X_1^{a_1} \ldots X_d^{a_d}$, with all $a_i$'s nonzero and divisible by 
$q-1$. Since $d \geq 2$, these do not occur, and the claim is proved.
\end{dem}

\begin{rem}
Recall that $\Gamma^q(V)$ is spanned by pure symbols $[v]_q,$ for $v\in V$. On these, it is straightforward that 
\[T_{H,V}: \Gamma^q_k(V) \to H\]
 is given by the formula
\[
    [v]_q \mapsto
    \begin{cases}
      0 , & \text{if $v \notin H$} \\
      v,       &  \text{if $v \in H$. }
    \end{cases}
 \]
It is (perhaps) surprising, that  such a formula actually defines a $k$-linear map! One may then provide another proof of Lemma \ref{lemTFH} using the following fact: given  $0 \neq v \in V$, the number of hyperplanes $H$ containing $v$, is congruent to $1$ mod $q$ (hence mod $p$). 
\end{rem}

    \begin{lem}\label{lemTindep}
    Let $0 \subsetneq  W  \subsetneq V$ be a $k$-subspace, of codimension $c$.  \\Choose a complete flag $\nabla$ on the $k$-vector space $V/W$, yielding a filtration \[W=V_{d-c} \subset V_{d-c+1} \subset \ldots \subset V_d=V,\] with $\dim(V_i)=i$. 
    
    Then, the  composite \[T_\nabla:=(T_{V_{d-c},V_{d-c+1}}  \circ T_{V_{d-c+1},V_{d-c+2}} \circ \ldots \circ T_{V_{d-1},V_d}): \A(V) \to \A(W)\] is a homogeneous law, of degree $q^c$. It does not depend on the choice of $\nabla$.

\end{lem}

\begin{dem}

    By definition, it is clear that $T_\nabla$ is   of the shape \[v \mapsto \sum_{i=0}^c \lambda_i(v) F^i(v),\] for some $k$-morphisms \[\lambda_i: \A(V) \to \A^1,\] in other words, elements of the symmetric algebra $A:=\Sym_k(V^\vee)$. Consider the following (totally split) polynomial with coefficients in $A$: \[P(T):= \prod_{\pi \in (V/W)^\vee}(T- \pi),\] where $(V/W)^\vee$ is considered as a $k$-subspace of $V^\vee =\Sym^1(V^\vee) \subset A$, in the natural way. We claim that \[T_{\nabla}(v)=P(F)(v). \]As the right side is clearly independent of $\nabla$, this proves the Lemma. The claim is easily verified if $c=1$, i.e. if  $W=H$ is a hyperplane, for then $P(X)=X^q-\pi^{q-1}X$, where $0 \neq \pi$ is any linear form with kernel $H$ (as in Definition \ref{defiTHV}). The general case in then readily checked, by induction on $c$.
\end{dem}

   \begin{defi}\label{defiTWC}
 Let $0 \subsetneq  W  \subsetneq V$ be a $k$-subspace. Denote by $T_{W,V}$ the morphism $T_\nabla$ of Lemma \ref{lemTindep} (which does not depend on the choice of a complete flag $\nabla$).

\end{defi}

	\begin{thm}\label{transfrob}
	
		Let $c \in \{1,\ldots, d-1\}$ be an integer.
Consider the composite $$\Phi_c:=\A_k(V) \stackrel {T_{c,V}} \to \A_k(\bigoplus_W W) \to \A_k(V),$$  where the sum is taken over all $k$-subspaces $W \subset V$ of codimension $c$, where $T_{c,V}$ denotes the sum of all $T_{W,V}$'s,  and where the second map is given by  sum of the inclusions $W\to V$.\\
	Then $\Phi_c=F_V^c$.

\end{thm}

\begin{dem}

By induction on $c \geq 1$. The case $c=1$ is Lemma \ref{lemTFH}. For the induction step, look at the composite  $$ \A_k(V) \stackrel  {T_{c,V}} \to \A_k(\bigoplus_{Z \subset V} Z)  \stackrel  {\sum T_Z}  \to \A_k(\bigoplus_{W \subset Z \subset V} W) \xrightarrow{\nat} \A_k(\bigoplus_{Z \subset V} Z) \xrightarrow{\nat} \A_k(V), $$  where $\mathrm{codim}(W)=c+1$, $\mathrm{codim}(Z)=c$, and  the middle  sum is taken over all partial flags $(W \subset Z \subset V)$.  Let us compute it, in two different ways. \\ First way. The composite of the two arrows in the middle, equal $F=\frob^r$ by Lemma \ref{lemTFH} (applied to each factor $Z$). By naturality of frobenius, it follows   that the composite of all four arrows equals the composition of $\Phi_c$ and $F_V$, hence equals $F_V^{c+1}$ by the induction hypothesis.\\
Second way. For every partial flag $W \subset Z \subset  V$ as above, the composite \[\A_k(V) \stackrel  {T_{Z,V}}  \to  \A_k(Z) \stackrel  {T_{W,Z}}  \to \A_k(W)  \] equals $T_{W,V}$ (Lemma \ref{lemTindep}). Observe that, for a given subspace $W\subset V$, the number of possible choices for $Z $ is  congruent to $1$ modulo $p$ (the cardinality of a projective space over  $k$). Thus, the composite of all four arrows  equals $\Phi_{c+1}$.

\end{dem}

\begin{rem}
Let $G$ be a (pro)finite group, and assume that $V$ is a $(k,G)$-module. In Theorem \ref{transfrob}, take $c=d-1$. One then obtains a factorisation of $F_V^{d-1}$ as, $$\A_k(V) \stackrel {T_{d-1,V}} \to \A_k(\bigoplus_L L) \to \A_k(V).$$ Upon linearizing this factorisation (using the universal property of divided powers), then dualizing, this yields the following. There exists a canonical linear map $T$, such that  the  $k$-linear map (verschiebung) 
$$\fonction{\ver^{r(d-1)}}{V}{\Sym^{q^{d-1}}(V)}{v}{v^{q^{d-1}}}$$
factors as  as \[ V\xrightarrow{\nat}  \bigoplus_{V \to L} L \xrightarrow{T} \Sym^{q^{d-1}}(V) , \] where the direct sum is taken over all $k$-lines $L$ that are \textit{quotients} of $V$. If $G$ is a $p$-group, then the middle term is a permutation $G$-module, in the sense of Definition \ref{defipermmodule}.  One may use this  to give a refined  version of  Theorem \ref{FIT}, where   the permutation module, the integer $m$ and the  factorization are explicit.
\end{rem}

\begin{exo}
    As a continuation of the preceding Remark, give a short proof of Theorem \ref{FIT}, using Theorem \ref{transfrob}.
    \end{exo}
\section{Proof of Theorem A, for $S$ affine and $L = \mathcal O_S$.}

To keep notation light,  we give the proof for $n=1$. The general case is the same.

We first assume that $S=\Spec(A)$ and $L = \mathcal O_S$. Note that in this case, any $(G,\W_r(L)(1))$-torsor is strongly geometrically trivial by Lemma \ref{AffGeomTriv}. By Proposition \ref{BuildGeomTriv}, $(G,\W_r(L)(1))$-torsors are then classified by $H^1(G,\W_r(A)(1))$- in the usual setting of the cohomology of a profinite group $G$, with values in a discrete $G$-module. We are then concerned with showing that after some suitable frobenius pullback, all such classes admit a compatible system of liftings.

To prove the Theorem, it is straightforward to reduce to the case where $A$ is an $\F_p$-algebra of finite-type. Indeed, consider a cohomology class $c$, represented by a cocycle $$z_g\in Z^1(G,\W_r(A)(1)),$$ which factors through an open subgroup of $G$. It thus  takes finitely many values. Each of these values has $r$ Witt coordinates (in $A$). The  $G$-orbit of each of these is also finite. Altogether, we may then replace $A$ by the $(G,\F_p)$-algebra generated by a finite $G$-invariant collection of elements of $A$.

In the current setting, Theorem A then boils down to the following Proposition. Its content is more precise: the growth rate of the power of frobenius needed to lift classes in $ H^1(G, \W_{r}(A)(1))$ is actually \textit{linear} in $r$.

\begin{prop} \label{LIFT1EXT}
Let $A$ be a $(G,\F_p)$-algebra of finite-type over $\F_p$, and let $(G,\Z/p^{1+e}(1))$ be a $(1,e)$-cyclotomic pair, with $e\in \N^{\ast} \cup \{\infty\}$. There is a non-negative integer $m(A)$ with the following property.

Let $r \in \{ 1, \ldots, e\}$ be an integer and $c \in H^1(G, \W_{r}(A)(1))$ be a cohomology class. Then $(\frob^{m(A)r})^*(c)$ lifts to $ H^1(G, \W_{e+1}(A)(1))$.
\end{prop}

\begin{proof}
 By Theorem \ref{FIT} there exists $m=m(A) \geq 0$ and a factorization \[\frob^{m}:A \stackrel f \to \F_p^{(X)} \stackrel g \to A, \] for some $G$-set $X$. We are going to show that this $m$ satisfies the conclusion of the Proposition. We first deal with the case $r=1$, showing that classes in the image of (the map induced on $H^1(G,\cdot)$ by the cyclotomic twist of) $g$ lift  to $ H^1(G, \W_{e+1}(A)(1))$. The $G$-set $X$ is a disjoint union of cosets $G/H_i$, where the $H_i$'s are open subgroups of $G$. It suffices to treat the case of a single orbit $G/H$. Using Shapiro's Lemma, we can then replace $G$ by $H$,  reducing to the case $X=\{* \}$. Put $$a:=g([*]) \in A.$$ For all $i \geq 1$, denote by $$a_{i+1}:=\tau_{i+1}(a) \in \W_{i+1}(A)$$ the Teichm\"uller representative of $a$.
 
 Let $0 \leq i \leq e$ be an integer. We have a commutative diagram
 
 \[\xymatrix @R=5mm{ \Z/p^{i+1} \Z  \ar[r] \ar[d] & \ldots \ar[r] & \Z/p^{2} \Z  \ar[r] \ar[d] & \Z/p \Z  \ar[d] \\ \W_{i+1} (A) \ar[r] & \ldots \ar[r] & \W_{2} (A) \ar[r] & A , } \]
 where the horizontal maps are the natural surjections, and the $i$-th vertical map sends $1 \in \Z/p^{i+1} \Z $ to $a_{i+1}$. After twisting this diagram by $\Z/p^{1+e}(1)$, by definition of $(1,e)$-smoothness, all arrows in the upper line induce surjections on $H^1(K,\cdot)$, for any open subgroup $K$ of $G$. Thus, $\Im(g_*) \subset H^1(G,A(1))$ indeed consists of classes, that lift as required.
 
The  general case is by induction on $r$. Assuming the result known for $r$, let $$c \in H^1(G,\W_{r+1}(A)(1))$$ be a cohomology class. Denote by $b$ its reduction to a class in $H^1(G,\W_{r}(A)(1))$. By induction, we know that $b_r:=(\frob^{rm})^*(b)$ admits a lifting $$(b_{1+e}) \in H^1(G,\W_{1+e}(A)(1)).$$ Denote by $(b_{r+1}) \in H^1(G,\W_{1+r}(A)(1)) $ the reduction of $b_{1+e}$. Set $$c':=(\frob^{rm})^*(c)-b_{r+1}.$$ Via the maps induced in cohomology from the exact sequence \[ 0 \to A(1) \stackrel {i_r} \to \W_{r+1}(A)(1) \to \W_{r}(A)(1) \to 0,\] $c'$ reduces to $0$ in $H^1(G,\W_{r}(A)(1))$, hence comes from a class $b' \in H^1(G,A(1)).$ By the $n=1$ case, we get that $(\frob^m)^*(b')$ lifts to $H^1(G,\W_{1+e}(A)(1))$ . Hence, $(\frob^m)^*(c')$ lifts as well. Finally, we see that $$(\frob^{(r+1)m})^*(c)=(\frob^m)^*(b_{r+1})+(\frob^m)^*(c')$$ lifts as stated- as a sum of classes sharing this property. 
\end{proof}

\subsection{The general case.}\label{secproofAgen}
We now prove Theorem A, for $S$ and $L$ arbitrary.\\
By assumption, there exists a (not necessarily $G$-equivariant) trivialization \[ F: P \stackrel \sim \to \W_r(L)(1)\] of the $ \W_r(L)(1)$-torsor $P$ over $S$. Remembering that the automorphism group of the trivial $ \W_r(L)(1)$-torsor is $H^0_{}(S, \W_r(L)(1))$, we see that the assignment \[\fonction{z}{G}{H^0_{}(S, \W_r(L)(1))}{g}{z_g:= F^{-1} \circ g \circ F \circ g^{-1}} \] is a $1$-cocycle. The $(G,\W_r(L)(1))$-torsor $P$ can be recovered as the twist of the trivial $(G,\W_r(L)(1))$-torsor by this cocycle. Denote by $$c \in H^1(G, H^0_{}(S,\W_r(L)(1)))$$ the cohomology class of $z$.\\ Lifting $P$ as required is then equivalent to lifting $c$ to $$c_{1+e} \in H^1(G, H^0_{}(S,\W_{1+e}(L)(1))).$$ Theorem A is then a consequence of the following Proposition. 

\begin{prop} \label{LIFT2EXTL}
Let $e\in  \N_{\ast} \cup \{\infty\}$. Let $(G,\Z/p^{1+e}(1))$ be a $(1,e)$-cyclotomic pair. Let $S$ be a $(G,\F_p)$-scheme. Pick an integer $1\leq r \leq e$.

Let $$c \in H^1(G, H^0_{}(S,\W_r(L))(1))$$ be a cohomology class. Then, there exists an integer $m\geq 0$, such that the class $$(\frob^m)^{\ast}(c) \in H^1(G, H^0_{}(S,\W_r(L^{\otimes p^m}))(1))$$ lifts to $$c_{1+e} \in H^1(G, H^0_{}(S,\W_{1+e}(L^{\otimes p^m}))(1)).$$ 
\end{prop}

\begin{proof}
By Proposition \ref{TeichWitt}, we have a commutative diagram, with exact rows
\[ \xymatrix@C-=0.5cm{ 0 \ar[r] & H^0(S,\W_{2+i}(L^{\otimes p^r})) \ar[r] \ar[d] & H^0(S,\W_{r+2+i}(L)) \ar[r] \ar[d] & H^0(S,\W_{r+1+i}(L)) \ar[r] \ar[d] & 0 \\ 0 \ar[r] & H^0(S,\W_{1+i}(L^{\otimes p^r})) \ar[r] \ar[d] & H^0(S,\W_{r+1+i}(L)) \ar[r] \ar[d] & H^0(S,\W_{r+i}(L)) \ar[r] \ar[d] & 0 \\ & \myvdots \ar[d]&\myvdots \ar[d] & \myvdots \ar[d] & \\ 0 \ar[r] & H^0(S,L^{\otimes p^r}) \ar[r]^i & H^0(S,\W_{r+1}(L)) \ar[r]^\pi & H^0(S,\W_{r}(L)) \ar[r] & 0,} \]

where frobenius pushforwards are dismissed for clarity.

We work in the cyclotomic twist of this diagram, to which we apply $H^1(G,.)$, and mimic the proof of Proposition \ref{LIFT1EXT}. By induction on $r$, we assume the result known for a given $r\geq 1$, and for all $L$. Let $$c \in H^1(G, H^0_{}(S,\W_{r+1}(L))(1))$$ be a cohomology class. Then, there exists $m_1 \geq 1$ such that $$\pi_*((\frob^{m_1})^*(c)) \in H^1(G, H^0_{}(S,\W_{r}(L^{\otimes p^{m_1}}))(1))$$ admits a compatible system of liftings $(b_i)_{r \leq i \leq e+1}$. Replacing $L$ by $L^{\otimes p^{m_1}},$ we can assume that $m_1=1$. Replacing $c$ by $c-b_{r+1}$, we then reduce to the case where $\pi_*(c)=0$. Hence, there exists $$a \in H^1(G, H^0_{}(S,L^{\otimes p^r}(1)))$$ such that $i_*(a)=c$. If we can show that (a high enough frobenius twist of) $a$ lifts to $$H^1(G,H^0(S,\W_{e+1-r}(L^{\otimes p^r})) )$$ (with respect to the line bundle $L^{\otimes p^r}$), then we are done, by commutativity of the diagram above.

Thus, only the case $r=1$ remains to be considered. Put \[ A:=\bigoplus_{ i \in \Z} H^0_{}(S, L^{\otimes i}) ; \] the $(\F_p,G)$-algebra of regular functions on the $\G_m$-torsor associated to $L$. As usual, the class $c \in H^1(G, H^0_{}(S,L)(1))$ is defined by a cocycle taking only finitely many values. Let $A' \subset A$ be the sub-$(\F_p,G)$-algebra generated by these values; it is an $\F_p$-algebra of finite-type. Casting Theorem \ref{FIT} again, we get an integer $m \geq 0$ and a factorization \[ \frob^m: A' \stackrel f \to \F_p^{(X)}\stackrel g \to A',\] where $X$ is a $G$-set. Consider the composite \[\phi: \F_p^{(X)}\stackrel g \to A' \stackrel {\subset} \to A \stackrel {\mathrm{pr}_m} \to H^0_{}(S,L^{\otimes p^m}),\] where $\mathrm{pr}_m$ is the natural projection. We are now reduced to showing that classes in the image of \[\phi(1)_*: H^1(G, \F_p^{(X)}(1)) \to H^1(G,H^0_{}(S,L^{\otimes p^m}(1))) \] lift to 
$H^1(G,H^0_{}(S,\W_{1+e}(L)^{\otimes p^m}(1)))$.\\
Note that the stabilizers of elements of $X$ are open subgroups of $G$. By definition of a cyclotomic pair, and using Shapiro's Lemma, we can  thus replace $G$ by each of these stabilizers. In short,  we can assume that $X=\{ *\}$ is a one-element set.\\Put $$a := \mathrm{pr}_m(g([*]) )\in H^0_{}(S,L^{\otimes p^m}).$$ For all $s \geq 1$, denote by $$a_s:= \tau_s(a) \in H^0_{}(S,\W_r(L^{\otimes p^m}))$$ the Teichm\"uller lift of $a=a_1.$ We conclude by a chase in the diagram \[\xymatrix @C=1pc { (\Z/p^{e+1} \Z )(1) \ar[d]^{1 \mapsto a_{1+e}} \ar[r] & (\Z/p \Z )(1) \ar[d]^{1 \mapsto a} \\ H^0_{}(S,\W_{e+1}(L^{\otimes p^m})(1)) \ar[r] & H^0_{}(S,L^{\otimes p^m}(1)) , } \] to which we apply the functor $H^1(G,\cdot)$-- remembering the definition of a cyclotomic pair.
\end{proof}

\section{Cyclothymic profinite groups and $(n,1)$-smoothness}

 The following notion is to be thought as a smooth profinite group $G$, with a moving cyclotomic character. It makes sense in arbitrary depth $e \geq 1$, but for simplicity, we assume $e=1$ throughout.

\begin{defi}(Cyclothymic profinite group.) \label{DefiCyclothymic}\\
Let $n\geq 1$, and let $G$ be a profinite group.\\
We say that $G$ is $(n,1)$-cyclothymic if the following holds.\\
Let $L$ be an $(\F_p,G)$-module, of dimension one as an $\F_p$-vector space. \\Consider a finite collection $$H_1,\ldots,H_N \subset G$$ of open subgroups of $G$ and for each $i=1,\ldots,N$, let $c_i \in H^n(H_i,L)$ be a cohomology class. Set $C:=(c_1,\ldots,c_N)$. \\Then, there exists a lift of $L$, to a $(\Z/p^2,G)$-module  $L_2[C]$, free of rank one as a $\Z/p^2$-module, such that, for each $i=1,\ldots,N$, the class $c_i$ lifts through $$H^n(H_i,L_2[C]) \to H^n(H_i,L).$$ 
\end{defi}
\begin{rem}
The point here is that $L_2[C]$ \textit{depends} on $C$.
\end{rem}

\begin{rem}
In the definition,  we can assume w.l.o.g. that $L=\F_p$ has the trivial action of $G$. Indeed, let  $G_L \subset G$  be the kernel of the action on $L$. As its index  it is prime-to-$p$,  we may replace $H_i$ by $H_i \cap G_L$-- using  a restriction-corestriction argument. Working out details is left to the interested reader.
\end{rem}

\begin{rem}
    If $(G,\Z/p^2(1))$ is an $(n,1)$-cyclotomic pair, then $G$ is $(n,1)$-cyclothymic. Indeed, by the preceding Remark, one may assume w.l.o.g. $L=\F_p=\F_p(1)$, so that one may take $L_2[C]:=\Z/p^2(n),$ independently of $C$. 

\end{rem}
Before investigating relations between cyclotomic pairs, cyclothymic and smooth profinite groups, we state the following cyclothymic variant of Theorem A. 

\begin{thm}[Cyclothymic version of Theorem A]\label{WeakLiftCyThy} \hfill\\ 
Pick  $n \in \N$, and let $G$ be a $(n,1)$-cyclothymic profinite group.\\
 Let $S$ be a $(G,\F_p)$-scheme, and $L$ be a $G$-linearized line bundle over $S$. Consider a strongly geometrically trivial class $$c \in H^n((G,S),L).$$ There exists $m \geq 0$, and a lift of the trivial $(\F_p,G)$-module $\F_p$, to a $(\Z/p^2,G)$-module $\Z/p^2[c]$,  free of rank one as a $\Z/p^2$-module,  such that $c^{(m)}$ lifts  via 
\[ H^n_{sgt}((G,S),\W_2(L^{\otimes p^m})\otimes_{\Z}\Z/p^2[c]) \to H_{sgt}^n((G,S),L^{\otimes p^m}). \] 
In particular, the group $G$ is $(n,1)$-smooth  (take for $S$  a perfect affine scheme).
 \end{thm}

\begin{proof}
The proof follows the same lines as the one of Theorem A. To understand why, the key is that, in the proof of Theorem A for $r=1$, it suffices to lift a \emph{finite} number  $N$ of classes $c_i \in H^1(H_i,\F_p)$, where $H_i \subset G$ are open subgroups. These $H_i$'s occur as the stabilizers of elements of the $G$-set $X$, given by Theorem \ref{FIT}.  The coefficients module used to lift the $c_i$'s  is of little importance--provided it is a $(\Z/p^2,G)$-module, free of rank one as a $\Z/p^2$-module. For this purpose, setting $C:=(c_1,\ldots,c_N)$, the module $\Z/p^2[c]:=L_2[C]$ of Definition \ref{DefiCyclothymic} does the job,  in place of $\Z/p^2(n)$ used in the cyclotomic case.
\end{proof}

We  get to the main result of this section.

\begin{thm}\label{Smooth=Cyclothymic}
Let $G$ be a profinite group. Pick $n\geq 1$.\\
Then, $G$ is $(n,1)$-cyclothymic if and only if it is $(n,1)$-smooth.
\end{thm}

\begin{proof}
The first implication is contained in Theorem \ref{WeakLiftCyThy}.\\
 We prove the converse implication when $n=1$, the general case being identical. Let $H_1,\ldots,H_k \subset G$ be open subgroups, and let $$\chi=(\chi_1,...,\chi_k)\in \prod_{i=1}^k H^1(H_i,\mathbb{F}_p)$$ be cohomology classes (characters). Introduce $$A:=\mathbb{F}_p[X_{i,c}],$$ the polynomial algebra on $d=\sum_{i=1}^k \vert G/H_i \vert$ variables, indexed by $i=1,\ldots,k$ and $c \in G/H_i$. For each fixed $i$, the group $G$ naturally permutes the variables $X_{i,c}$, $c\in G/H_i$, allowing to view $A$ as an $(\F_p,G)$-algebra.
 
 Using Shapiro's Lemma, the $\chi_i$'s give rise to $1$-cocycles $$\xi_i: G \to \bigoplus_{c \in G/H_i} \F_p X_{i,c}$$ depending, up to a coboundary, on the choice of a system of representatives of the factor set $G/H_i$. We then form the $1$-cocycle $$ \xi:=\sum_{i=1}^k \xi_i: G \longrightarrow A.$$
As $G$ is $(1,1)$-smooth, there is an integer $m$ and a lift of the $(A,G)$-module $A$, to a $(\W_{2}(A),G)$-module $$\W_{2}(A)(\xi^{(m)}),$$ free of rank one as a $\W_{2}(A)$-module, such that $\xi^{(m)}$ lifts to $H^1(G,\W_{2}(A)(\xi^{(m)}))$. Consider the extension of $(\W_{2}(A),G)$-modules $$ \mathcal E: 0 \to A \to \W_2(A) (\xi^{(m)}) \stackrel \pi \to A \to 0. $$ 
Let $i=1,\ldots k$ be an integer. Let $X^{\alpha} \in A$ be a pure monomial, in the variables $X_{i,c}$. The inclusion \[ \iota_{ \alpha}: \F_p X^{\alpha} \to A,\] is then naturally split by the projection \[ \epsilon_{\alpha}: A \to \F_p X^{\alpha}.\] Setting $H_\alpha \subset G$ to be the stabilizer of $\alpha$, it is clear that these arrows are $H_\alpha$-equivariant.\\ If $X^{\beta} \in A$ is another pure monomial, we can form the extension of
 $\Z/p^2$-modules $$ \mathcal F_{\alpha,\beta}:=(\epsilon_{\beta})_*(\iota_{\alpha}^*(\mathcal E)) : 0 \to \F_p X^\beta \to F_{\alpha,\beta} \to \F_p X^\alpha \to 0. $$
We are going to describe its middle term $ F_{\alpha,\beta}$. To do so, consider the commutative diagram of $\Z/p^2$-modules
 
 \[ \xymatrix{\iota_{0}^*(\mathcal E):0 \ar[r] & A \ar[r] \ar[d]^{X^{p\alpha}.} & E_0 \ar[r] \ar[d]^{ \tau_2(X^\alpha).} 
& \F_p \ar[r] \ar[d]^{ X^\alpha.} & 0 \\ \iota_{\alpha}^* (\mathcal E): 0 \ar[r] & A \ar[r] & E_{\alpha} \ar[r] & \F_p X^{\alpha}\ar[r] & 0, }\] where $\tau_2(.) \in \W_2(A)$ denotes the multiplicative representative. We infer a natural isomorphism of extensions $\Z/p^2$-modules $$\mathcal F_{\alpha,\beta} \simeq (\epsilon_{\beta}( X^{p\alpha}.))_*(\iota_{0}^*(\mathcal E)).$$The arrow $\epsilon_{\beta}( X^{p\alpha}.)$ vanishes if $p\alpha$ does not divide $\beta$. In that case, we deduce that $\mathcal F_{\alpha,\beta} $ has a canonical splitting. In particular, it is a trivial extension of $(\F_p,H_{\alpha} \cap H_{\beta})$-modules.

If $\beta=p\alpha$, the  arrow $\epsilon_{\beta}( X^{p\alpha}.)$ factors through $\epsilon_0$, yielding a canonical isomorphism $$ \mathcal F_{\alpha,{p\alpha}}\simeq \mathcal F_{0,0}. $$As a $\Z/p^2$-module, $F_{0,0}$ is free of rank one. We put $$\Z/p^2[\xi]:=F_{0,0}.$$
If $\beta\neq p\alpha$ and $p\alpha$ divides $\beta$, the extension $\mathcal F_{\alpha,\beta} $ is an extension of $(\F_p,H_{\alpha} \cap H_{\beta})$-modules, which may be non-trivial.

For $i=1,\ldots,k$, denote by $$\epsilon_i: A \to \bigoplus_{c \in G/H_i} \F_p X_{i,c}^{p^{m+1}}$$ the projection, i.e. the sum of all arrows $\epsilon_{X_{i,c}^{p^{m+1}}}$. Similarly, consider $$\iota_i: \bigoplus_{c \in G/H_i} \F_p X_{i,c}^{p^{m}} \to A.$$ The arrows $\epsilon_i$ and $\alpha_i$ are $G$-equivariant. Form the extension of
 $(\Z/p^2,G)$-modules $$ \mathcal E_{i,j}:=(\epsilon_{j})_*(\iota_{i}^*(\mathcal E)) : 0 \to \bigoplus_{c \in G/H_j} \F_p X_{j,c}^{p^{m+1}} \to E_{i,j} \to \bigoplus_{c \in G/H_i} \F_p X_{i,c}^{p^{m}} \to 0. $$
If $i \neq j$, from what precedes, we get that $ \mathcal E_{i,j}$ has a (canonical, hence $G$-equivariant) splitting. Indeed,  no $X_{i,c}^{p^{m}} $ divides a $X_{j,c'}^{p^{m+1}} $. Similarly, $ \mathcal E_{i,i}$ is canonically isomorphic to $$ \mathcal F_{0,0}^{G/H_i} : 0 \to  \F_p^{G/H_i}\to \Z/p^2[\xi]^{G/H_i} \to \F_p^{G/H_i} \to 0. $$
Since $\xi^{(m)}$ lifts via (the map induced on $H^1(G,.)$ by) the surjection $\pi$ of $\mathcal E$, we deduce that it also lifts via the surjection of the extension of   $(\Z/p^2,G)$-modules $$\mathcal E':=\bigoplus_{i,j} \mathcal E_{i,j}, $$reading as $$\mathcal E': 0 \to \bigoplus_{j=1,\ldots,k; \hspace{2pt} c \in G/H_j} \F_p X_{j,c}^{p^{m+1}} \to E' \to \bigoplus_{i=1,\ldots,k; \hspace{2pt} c \in G/H_i} \F_p X_{i,c}^{p^{m}} \to 0. $$
We have  shown before that $$\mathcal E'=\bigoplus_{i} \mathcal E_{i,i}$$ is "diagonal", therefore each $\xi_i^{(m)}$ lifts via (the surjection of) $$ \mathcal E_{i,i} = \mathcal F_{0,0}^{G/H_i}. $$Equivalently, using Shapiro's Lemma, we conclude that each $\chi_i$ lifts to $H^1(H_i,\Z/p^2[\xi])$. 
\end{proof}

The following connections provided in this article will be used in the third part of this series, in the proof of the smoothness theorem.

\vspace{0.5cm}
{\centering\noindent\makebox[-5pt]

\begin{tikzpicture}[thick]
\node[draw,rectangle,align=center,right of=k] (a) at (4.5,0) {$(G,\mathbb{Z}/p^{2}(1))$ is a $(n,1)$-cyclotomic pair};
\node[draw,rectangle] (d) at (9.7,-1.5) {$G$ is $(n,1)$-smooth};
\node[] (d') at (10.7,-0.55) {\small{Theorem A}};
\node[draw,rectangle,align=center,below of=k] (c) at (1.25,-0.5) {$G$ is $(n,1)$-cyclothymic};
\node[] (c') at (0.1,-0.55) {\small{Remark 12.4}};
\node[] (x) at (5.6,-1.2) {\small{Theorems 12.5 and 12.6}};

  \draw[vec] (a) to (0.948,0);
  \draw[vecArrow] (1.25,0.025) to (c);
  \draw[vec] (a) to (10.002,0);
  \draw[vecArrow] (9.7,0.025) to (d);
  \draw[vecArrow] (5,-1.50) to (8.17,-1.50);
  \draw[vecArrow] (7,-1.5) to (3.13,-1.5);
  \draw[innerWhite] (1.22,0) to (2.45,0);
  \draw[innerWhite] (8.7,0) to (9.883,0);
\end{tikzpicture}
}

\section{Necessity of the perfectness assumption.}\label{secperfnec}
We assume $n=e=1$ in this section.
\begin{qu}
  In Theorem A,  is the frobenius twist $(.)^{(m)}$ necessary?
\end{qu} The answer is `yes, by all means!', as  explained next.

\begin{defi}\label{deficovid}
    Let $G$ be a profinite group. Say that \begin{center}
     ``the cohomology of every  $(\F_p,G)$-module $V$ is induced from dimension 1''    
    \end{center}(relative to the prime $p$) if the following property holds.\\
    For every  $(\F_p,G)$-bundle $V$, the surjection of $(\F_p,G)$-bundles 
    $$\begin{array}{ccc}
 \F_p[V] & \overset{\sigma}{\longrightarrow} & V \\
 \sum_{v \in V} a_v [v] & \longmapsto & \sum_{v \in V} a_v v \end{array}$$
induces a surjection \[ H^1(G,\F_p[V]) \xrightarrow{H^1(\sigma)} H^1(G,V).\]

    In short, if that property  holds, we say that $G$ has covid1.
\end{defi}

\begin{rem}
    In Definition \ref{deficovid}, the $(\F_p,G)$-module $\F_p[V]$ is permutation, so that by Shapiro's Lemma, $H^1(G,\F_p[V])$ is isomorphic to a finite direct sum of groups  $H^1(H_i,\F_p)$. [Precisely, the open subgroups $H_i \subset G$ range through the stabilisers of a system of representatives of the orbit space $V/G$.] This justifies the terminology `induced from dimension 1'.
\end{rem}
\begin{exo}
    Show that $G$ has covid1, if and only if the following holds.  For every  $(\F_p,G)$-bundle $V$, and for every $c \in H^1(G,V)$, there exists a permutation  $(\F_p,G)$-bundle $P$, and a morphism of $(\F_p,G)$-bundles $P \xrightarrow{s} G$, such that $c$ belongs to the image of $H^1(G,P) \xrightarrow{H^1(s)} H^1(G,V)$.
\end{exo}
\begin{rem}\label{RemcovH}
    If $G$ has covid1, so does every open subgroup $H \subset G$. This is  straightforward using Shapiro's Lemma, yielding, for every  $(\F_p,H)$-module  $V$, a natural isomorphism $H^1(H,V)=H^1(G, \Ind_H ^G(V))$.

\end{rem}

\begin{prop}\label{propnoperfcov}
   Let $G$ be a $(1,1)$-smooth profinite group. Assume that Theorem A (for $n=e=1$) holds with $m=0$, i.e. that  no frobenius twist is needed, prior to lifting. Then  $G$ has covid1.
\end{prop}

\begin{dem}
   Assume that Theorem A holds  with $m=0$. Let $V$ be a $d$-dimensional  $(\F_p,G)$-bundle. Let us show by induction on $d$, that the arrow $H^1(\sigma)$ of Definition \ref{deficovid} is surjective. This is obvious if $d=1$. Assume surjectivity holds in dimension $d-1$. Pick $c \in H^1(G,V)$.  Introduce the  $(G,\F_p)$-algebra $A:=\Sym_{\F_p}^*(V)$, so that $\Spec(A)=\A_{\F_p}(V^\vee)$. There is the tautological $A$-linear map \[ f: V \otimes_{\F_p} A \longrightarrow A.\] Define the class $(c \otimes 1) \in H^1(G, V \otimes_{\F_p} A$), and \[ C:=f_*(c \otimes 1) \in H^1(G,A).\]  Consider the natural reduction extension \[\mathcal R: 0 \longrightarrow A \xrightarrow{\ver} \W_2(A) \overset{\rho}{\longrightarrow} A \longrightarrow 0.\]  By  assumption, $C$ lifts  via $\rho_*$ (without the need for a frobenius twist).  Apply the following operations to the extension $\mathcal R$: push-forward  (on the left) with respect to the projection $A \longrightarrow \Sym_{\F_p} ^p(V)$, and pull-back (on the right) with respect to the inclusion $V=\Sym_{\F_p} ^1(V) \overset{\iota}{\longrightarrow} A$. The resulting extension reads as \[ 0 \longrightarrow \Sym_{\F_p} ^p(V)  \longrightarrow \Gamma^p_{\Z/p^2}(V) \overset{\phi}{\longrightarrow} V \longrightarrow 0,\] where the middle term is the $p$-th divided power of $V$, \textit{regarded as a $(\Z/p^2)$-module.} The map $\phi$ is given, on pure symbols, by $\phi([v]_p)=v$.  Such phenomena were studied systematically in the preprint \cite{DCFOr}.  The definition of $C=\iota_*(c)$, combined with a diagram chase, shows that $c$ lifts via $\phi_*$. Observe that $\phi$ factors as \[  \Gamma^p_{\Z/p^2}(V) \xrightarrow{can}  \Gamma^p_{\F_p}(V) \xrightarrow{F_V} V.\] As a consequence, $c$ lifts via $(F_V)_*$, to an element of $H^1(G,  \Gamma^p_{\F_p}(V))$. By Lemma \ref{lemTFH} (for $k=\F_p$), $F_V$ factors as \[ \Gamma_{\F_p}^p(V) \xrightarrow{T_V} \bigoplus_{H \subset V} H \xrightarrow{\nat} V,\] where the direct sum is over all $\F_p$-hyperplanes of $V$ (in other words, over all $\F_p$-rational points of  $\P(V)$), where $T_V$ is the sum of all transfers $T_{H,V}$, and  where $\nat$ is the sum of the inclusions $H \subset V$. Observe that this is a factorization in the category of $(\F_p,G)$-bundles:  the middle term is a direct sum of induced $(\F_p,G)$-bundles, and the arrows $T_V$, $\nat$ are clearly $G$-equivariant. \\
Since $c$ lifts via $(F_V)_*$, \textit{a fortiori} it lifts via $\nat_*$. By Shapiro's Lemma, the cohomology group $H^1(G,  \bigoplus_{H \subset V} H)$ is a direct sum of groups of the shape $H^1(G_i, H_i)$,  where the $H_i$'s run through a system of representatives of the coset space $\P(V)(\F_p)/G$, and $G_i:=\mathrm{Stab}(H_i) \subset G$.  In that way, to show that  $ H^1(G,V)$ is induced from dimension one, it suffices to show that each  $H^1(G_i, H_i)$ has that property, too. As open subgroups of $G$, the $G_i$'s  still satisfy the `no frobenius' stronger form of Theorem A.  Conclude using the induction assumption.\end{dem}

\begin{prop}\label{propcovfree}
    Assume that $p$ is odd. Let    $G$  be  any profinite group.\\ Then $G$ has covid1 if and only if the $p$-cohomological dimension  of $G$ is  $\leq 1$.
\end{prop}
\begin{dem} 

 If  the $p$-cohomological dimension  of $G$ is  $\leq 1$, then $G$ has covid1  (use that $H^2(G,\Ker(\sigma))=0$ in definition \ref{deficovid}). Let us prove the converse.  Observe that, since $p$ is odd, the cup-product \[H^1(G,\F_p) \times H^1(G,\F_p) \longrightarrow  H^2(G,\F_p) \] is alternating. We need to show that, for every open subgroup $H \subset G$, the group $H^2(H,\F_p)$ vanishes.  Thanks to Remark \ref{RemcovH}, it suffices to prove this for $H=G$. The  first step is to show that, for every $e_0,e_1 \in H^1(G,\F_p)$, \[e_0 \cup e_1 \stackrel ? = 0 \in H^2(G,\F_p).\] If $e_0, e_1$ are $\F_p$-collinear,  this holds because the cup-product is alternating.  Otherwise, set $H_i:=\Ker(e_i)$ and $H:=H_0 \cap H_1$. Observe that $G/H_0 \simeq G/H_1 \simeq \Z/p$, and $G/H \simeq (\Z/p)^2$. Let $H_0,H_1,\ldots, H_p$ denote the $p+1$ subgroups, such that $H \subsetneq H_i \subsetneq G$.  Pick $e_2, \ldots, e_p \in H^1(G, \F_p)$, such that $H_i=\Ker(e_i)$. Define the exact sequence of $(\F_p,G)$-bundles \[(E)= 0 \longrightarrow  \F_p \overset{\iota}{\longrightarrow} \bigoplus_{i=0}^p \F_p[G/H_i] \longrightarrow V \longrightarrow 0,\] by requiring that $\iota$ is the direct sum of the norms 
 $$\fonction{N_i}{ \F_p}{ \F_p[G/H_i]}{1}{\sum_{x \in G/H_i} [x]}$$
Denote by $\beta_E$ the connecting arrow of $(E)$ in cohomology.

 By Shapiro's Lemma, for $i=0,\ldots,p$, \[(N_i)_*(e_0 \cup e_1)= \Res_{H_i}(e_0) \cup\Res_{H_i}(e_1)  \in H^2(H_i,F_p).\] One readily checks, that $\Res_{H_i}(e_0), \Res_{H_i}(e_1)  \in H^1(H_i,F_p)$ are collinear, so that $(N_i)_*(e_0 \cup e_1)=0$. Consequently, $\iota_*(e_0 \cup e_1)=0$, and  there exists $c \in H^1(G,V)$, with $\beta_E(c)=e_0 \cup e_1$. Let us form the pull-back diagram

 \[ \xymatrix{ (F): 0 \ar[r]   &  \F_p \ar[r] \ar@{=}[d] &  \ast \ar[r] \ar[d] & \F_p[V]  \ar[r]  \ar[d]^\sigma & 0 \\ (E): 0 \ar[r]   &  \F_p \ar[r]^-\iota  &   \bigoplus_{i=0}^p \F_p[G/H_i] \ar[r]^-\pi  & V  \ar[r]  & 0. } \] Since $G$ has covid1, $c \in \Im(H^1(\sigma))$. By a straightforward chase, one sees that \[(e_0 \cup e_1) \in \Im( H^1(G, \F_p[V]) \xrightarrow{\beta_F} H^2(G,\F_p)). \] Vanishing of $(e_0 \cup e_1)$ will follow, if we show that the extension of $(\F_p,G)$-modules $(F)$ \textit{splits}. Because the $(\F_p,G)$-module $ \F_p[V]$ is permutation,  Shapiro's Lemma implies that   $(F)$  splits, if and only if the following property holds. For every open subgroup $K \subset G$, the natural arrow \[ H^0(K, \bigoplus_{i=0}^p \F_p[G/H_i] ) \xrightarrow{H^0(\pi)}   H^0(K, V) \] is \textit{surjective}. [This is also equivalent, to the existence of a \textit{set-theoretic} $G$-equivariant splitting of $\pi$.]\\
 Let us show an equivalent property: for every open subgroup $K \subset G$, the  arrow \[  H^1(K, \F_p) \xrightarrow{H^1(\iota)}  H^1(K, \bigoplus_{i=0}^p \F_p[G/H_i] ) \] is \textit{injective}.  Fix such a $K$. By Shapiro's Lemma again (and using that the $H_i$'s are normal in $G$), $\Ker(H^1(\iota))$  is  the intersection of the kernels of the restrictions \[  H^1(K, \F_p) \xrightarrow{\Res_{K \cap H_i}^K}   H^1(K \cap H_i,\F_p)  .\] If $K \subset H_i$ for some $i$,  then clearly $\Ker(H^1(\iota))=1$. Otherwise,  all $(p+1)$ normal subgroups $K \cap H_i \subsetneq K$ are of index $p$. In that case, replacing $G$ by $K$ and $H_i$ by $ K \cap H_i $, one may assume w.l.o.g. that $K=G$.  Then, the kernel of $\Res_{ H_i}^G$ is the one-dimensional $\F_p$-subspace $\F_p e_i$. The intersection of (any two of) these is trivial, yielding the sought-for injectivity.\\
 We have shown that the cup-product   \[H^1(G,\F_p) \times H^1(G,\F_p) \longrightarrow  H^2(G,\F_p) \] identically vanishes. Thanks to Remark \ref{RemcovH},  the same holds, upon replacing $G$ by an  open subgroup  $K \subset G$.  To  conclude, reset notation and let $c \in H^2(G, \F_p)$ be an arbitrary class.   Pick an exact sequence of $(\F_p,G)$-bundles \[(E)= 0 \longrightarrow  \F_p \overset{\iota}{\longrightarrow} Z \longrightarrow V \longrightarrow 0,\]  such that  $H^2(\iota)$ kills $c$. Form the pull-back diagram
 
 \[ \xymatrix{ (F): 0 \ar[r]   &  \F_p \ar[r] \ar@{=}[d] &  \ast \ar[r] \ar[d] & \F_p[V]  \ar[r]  \ar[d]^\sigma & 0 \\ (E): 0 \ar[r]   &  \F_p \ar[r]^-\iota  &  Z \ar[r]  & V  \ar[r]  & 0. } \]  Arguing as before, one sees that $c\in \Im (\beta_F: H^1(G,\F_p[V]) \to H^2(G,\F_p))$.\\ By Shapiro's Lemma, this image is spanned by elements of the shape \[\Cor_K ^G ( a \cup b ) ,\] for (various) open subgroups $K \subset G$, and elements $a,b \in H^1(K,\F_p)$. By the above, such cup-products vanish. The proof is complete.
\end{dem}

\begin{rem}
 If $p=2$, the equivalence of Proposition \ref{propcovfree} is false. A  counter-example is $\Z/2$, that has covid1 (because every $(\F_2,\Z/2)$-module is permutation), but  infinite $2$-cohomological dimension. However, the equivalence then holds, under the  assumption that the pair $(G,\Z/4)$ is $(1,1)$-cyclotomic. Indeed,  this ensures that the cup-product pairing is  alternating (see Exercise \ref{exopurep}), and the same proof goes through.
\end{rem}
Combining Propositions \ref{propnoperfcov} and \ref{propcovfree}, one thus gets:
\begin{prop}
  Let $p$ be odd, and let $G$ be an absolute Galois group, whose $p$-cohomological dimension is $\geq 2$. Then, in Theorem A, the   frobenius twist $(.)^{(m)}$ cannot be dismissed.
\end{prop} 

Let $G$ be a `reasonable' $(1,1)$-smooth profinite group. It may be interesting to  study the `growth rate' of the minimal value of $m$, in Theorem A. [Of course, the first step would be to clarify the meaning of `reasonable' and `growth rate'.] \\ Demushkin groups may be the case to start with.

\section*{Appendix: variations on $(n,1)$-smoothness}

In  this appendix, we provide some equivalent definitions of smoothness, which will be used in the next two articles of this series. First, we observe that the perfectness assumption on $A$, appearing in the definition of $(n,e)$-smoothness (Definition \ref{DefiSmooth}), can be removed if $e< \infty$, at the cost of introducing frobenius twists. This formally follows from the existence of the perfection $$A^{perf}:=\varinjlim_n A_n,$$ where $A_n=A$ for all $n$, transition morphisms being $\frob$, for any $\F_p$-algebra $A$, and from the isomorphisms $$\varinjlim_n \W_{1+e}(A_n) \overset{\sim}{\to} \W_{1+e}(A^{perf})$$ for $e< \infty$, and from the commutation between cohomology and direct limits. We thus get another equivalent  definition, for smooth profinite groups of finite depth.

\begin{defi}\label{DefiSmooth2}
Let $n\geq 1$ and $e\in \mathbb{N}$. A profinite group $G$ is $(n,e)$-smooth if and only if the following holds. Let $A$ be an $(\F_p,G)$-algebra and let $L_1$ be a locally free $A$-module of rank one, equipped with a semi-linear action of $G$. Let $$c \in H^n(G,L_1)$$ be a cohomology class. Then, there exists  $m \geq 0$ with the following property.

There exists a lift of $L_1^{(m)}$, to a $(\W_{e+1}(A),G)$-module $L_{e+1}^{[m]}[c]$, invertible as a $\W_{e+1}(A)$-module (and depending on $c$), such that $\frob^m(c)$ belongs to the image of the natural map \[H^n(G,L_{e+1}^{[m]} [c]) \to H^n(G,L_1^{(m)}). \]
\end{defi}
Another alternative definition of smooth profinite groups, is in terms of liftability of one-dimensional $G$-affine spaces (see section 4).

\begin{defi}[$(1,e)$-smooth profinite group, another equivalent definition]
Let $e\geq 1$ be an integer. A profinite group $G$ is $(1,e)$-smooth if and only if the following lifting property holds.\\Let $A$ be a perfect $(\mathbb{F}_p,G)$-algebra and let $X_1$ be a  $G$-affine space over $A$, such that $\overrightarrow{X_1}$ is an invertible $A$-module.\\
Then, $X_1$ admits a lift to a $G$-affine space $X_{1+e}$ over $\W_{1+e}(A)$, such that $\overrightarrow{X_{1+e}}$ is an invertible $\W_{1+e}(A)$-module. 
 \end{defi}

We  move on to some precisions about smooth profinite groups, in depth $1$.  These are not needed in \cite{F2} and \cite{DCF3}. Working out details is left to an interested reader.

\begin{prop}\label{propfree}
If $n=e=1$, we can add in Definition \ref{DefiSmooth} the extra requirement that $L_1=A$ is trivial, as a $(G,A)$-module. Therefore, $L_{2}[c]$ is geometrically isomorphic to $\W_{2}(A)$.
\end{prop}
\begin{dem}
The argument is similar to that of section \ref{secproofAgen} (where Theorem A is proved in the general case); in short: extending scalars to  the $\G_m$-torsor associated to the line bundle $L_1$.
\end{dem}

One may define  $(1,1)$-smoothness,  in the tongue of embedding problems.

\begin{defi}($(1,1)$-smooth profinite group, equivalent Definition)\label{DefiSmooth5}\\
 Denote by $\mathbf S \subset\mathbf{GL}_2$ one of the following two algebraic subgroups: the Borel subgroup $\mathbf B_2$, consisting of invertible matrices $$\begin{pmatrix}
* & *\\
0 & * 
\end{pmatrix},$$ or its subgroup $\Aut_{\mathrm{Aff}}( \A^1)=\G_a \rtimes \G_m \subset \mathbf B_2,$ consisting of invertible matrices $$\begin{pmatrix}
1 & *\\
0 & * 
\end{pmatrix}.$$ A profinite group $G$ is $(1,1)$-smooth iff the following lifting property holds. \\
Let $A$ be a perfect $(\F_p,G)$-algebra. Then, the natural map \[H^1(G,\mathbf S(\W_{2}(A))) \to H^1(G,\mathbf S(A))\] is onto.\\

\end{defi}
Let us briefly justify,  that this is indeed equivalent to Definition \ref{DefiSmooth}. There, by   Proposition \ref{propfree}, we may assume that the $A$-module $L_1$ is free of rank one. We  treat the least obvious case $\mathbf S=\mathbf B_2$. The datum of a cohomology class $b \in H^1(G,\mathbf B_2(A))$ is equivalent to an (isomorphy class of) extension of $(G,A)$-modules \[ \mathcal E_1: 0 \to D_1 \to E_1 \to D_1' \to 0,\] where $D_1$ and $D_1'$ are free of rank one as $A$-modules. The class of the extension \[\mathcal F_1:= \mathcal E_1 \otimes_A (D_1')^{-1}: 0 \to D_1 \otimes_A (D_1')^{-1} \to F_1:=E_1 \otimes_A (D_1')^{-1} \to A\to 0\] is an element of $H^1(G,L_1)$, where $$L_1:= D_1 \otimes_A (D_1')^{-1}.$$ Lifting $b$ as requested, amounts to lifting $\mathcal E_1$ to an extension of $(G,\W_{2}(A))$-modules \[ \mathcal E_{2}: 0 \to D_{2} \to E_{2} \to D'_{2} \to 0,\]where $D_{2}$ and $D_{2}'$ are free of rank one as $\W_{2}(A)$-modules. This is equivalent to lifting $\mathcal F_1$ to an extension \[\mathcal F_{2}: 0 \to L_{2} \to F_{2} \to \W_{2}(A)\to 0,\] where the $(G,\W_{2}(A))$-module $ L_{2}(=D_2 \otimes_{\W_2(A)} (D_2')^{-1})$, free of rank one as a $\W_{2}(A)$-module, of course depends on $b$. This liftability is equivalent to that of Definition \ref{DefiSmooth}.

\begin{rem}
A profinite group is $(1,1)$-smooth  if and only if its pro-$p$-Sylow subgroups are $(1,1)$-smooth.
\end{rem}
For a $(1,1)$-cyclotomic pair, the lifting property is required for all open subgroups $H \subset G$. This is not needed for $(1,1)$-smooth profinite groups:
 \begin{lem}
 Let $G$ be a $(1,1)$-smooth profinite group. Then, every closed subgroup $H \subset G$ is $(1,1)$-smooth as well.
 \end{lem}
 
 \begin{dem} By a standard limit argument, we can assume that $H$ is open in $G$. We use Definition \ref{DefiSmooth5}. Let $A$ be an $(\F_p,H)$-algebra. Consider the induced $(\F_p,G)$-algebra \[ \Ind_H^G(A):=\mathrm{Maps}_H(G,A),\] consisting of (left) $H$-equivariant maps $G \to A$, with ring structure induced by that of the target $A$. It is endowed with the natural $G$-action, given by the formula $(g.f)(x):=f(xg)$. We have $$\W_r(\Ind_H^G(A)) = \Ind_H^G(\W_r(A)),$$ since the formation of Witt vectors commutes to finite products. Thus, we have $$\mathbf B_2(\W_{2}(\Ind_H^G(A)))=\Ind_H^G(\mathbf B_2(\W_{2}(A))).$$ Shapiro's Lemma thus yields a natural bijection \[ H
^1(G,\mathbf B_2(\W_{2}(\Ind_H^G(A)))) \simeq H^1(H,\mathbf B_2(\W_{2}(A))),\] which we use to conclude that the arrow of Definition \ref{DefiSmooth5} is surjective for the pair $(H,A)$ iff it is for the pair $(G,\Ind_H^G(A))$.
 \end{dem}

\quad\\ \quad\\
\textsc{Acknowledgements.}
To be filled in.

\bibliographystyle{plain}
\bibliography{biblitex.bib}

\end{document}